\documentclass[reqno]{amsproc}

\usepackage{graphicx}
\usepackage{subfigure}
\usepackage{tikz, pgfplots}
\usepackage{hyperref}
\usepackage{epstopdf}

\hypersetup{
colorlinks=true,
linkcolor=blue,
anchorcolor=blue,
citecolor=blue
}

\usepackage{amssymb}
\usepackage{amsfonts}
\usepackage{amsmath}

\newtheorem{theorem}{Theorem}[section]
\newtheorem{lemma}[theorem]{Lemma}

\newtheorem{corollary}[theorem]{Corollary}

\newtheorem{conjecture}[theorem]{Conjecture}

\theoremstyle{definition}
\newtheorem{definition}[theorem]{Definition}

\newtheorem{remark}[theorem]{Remark}

\numberwithin{equation}{section}

\begin{document}
\title[Existence and Nonexistence of nontrivial positive solution]{Liouville's theorems to quasilinear differential inequalities
involving gradient nonlinearity term on manifolds}
\author[Sun]{Yuhua Sun}
\address{School of Mathematical Sciences and LPMC, Nankai University, 300071
Tianjin, P. R. China}
\email{sunyuhua@nankai.edu.cn}
\thanks{ Sun was supported by the National Natural Science Foundation of
China (No.11501303, No.11871296), and Tianjin Natural Science Foundation (No.19JCQNJC14600) }

\author[Xu]{Fanheng Xu}
\address{School of Mathematics (Zhuhai), Sun Yat-Sen University, 519082,
Zhuhai, P. R. China}
\email{xufh7@mail.sysu.edu.cn}
\thanks{Xu was supported by the Fundamental Research Funds for the Central Universities, SYSU (No.20lgpy154) }




\begin{abstract}
We investigate the nonexistence and existence of nontrivial positive solutions to
$\Delta_m u+u^p|\nabla u|^q\leq0$ on noncompact geodesically complete Riemannian manifolds, where $m>1$,
and $(p,q)\in \mathbb{R}^2$.
According to classification of $(p, q)$, we establish different volume growth conditions to obtain
Liouville's theorems for the above quasilinear differential inequalities, and we also show these volume growth conditions are sharp
in most cases.
Moreover, the results are completely new for $(p, q)$ of negative pair, even in
the Euclidean space.
\end{abstract}

\maketitle

\section{Introduction}
Let $M$ be a noncompact geodesically complete Riemannian manifold, and let us consider the following quasilinear elliptic differential inequality
\begin{equation}\label{equ}
\Delta_m u+u^p |\nabla u|^q \leq 0,\quad\mbox{on $M$},
\end{equation}
where $m>1$, $\Delta_mu=\text{div}(|\nabla u|^{m-2}\nabla u)$, and $(p, q)\in \mathbb{R}^2$, which means that $p, q$ can be allowed to be negative.

Throughout the paper, denote $d(x,y)$ the geodesic distance on
$M$, and $\mu$ the Riemannian measure on $M$. Set
\begin{align*}
B(x, r)=\left\{y\in M: d(x,y)<r\right\}.
\end{align*}
Let $o$ be a reference point on $M$. For our convenience, denote $B_r=B(o,r)$ and
\begin{equation*}\label{vol-def}
V(r)=\mu(B_r).
\end{equation*}

The paper is devoted to study the nonexistence and existence of nontrivial positive solutions to (\ref{equ}). Here
``nontrivial" means that the solution is not constant.
We aim to study how large $V(r)$ can be allowed to suffice that the Liouville's theorem holds
for nontrivial positive solutions to (\ref{equ}).  More precisely,
by classifying $(p, q)$, and then establishing different volume growth conditions even without imposing any curvature conditions
on $M$, we obtain the
nonexistence of
nontrivial positive solution to (\ref{equ}). Later, we also show our volume growth conditions are sharp in most cases.

Such related problems to (\ref{equ}) is well-studied in Euclidean space, for example, see the work of Mitidieri-Pohozaev \cite{MP99, MP01}, and Serrin-Zou \cite{SZ02}.
Though we investigate the same problems on manifolds, however our interest is different.
On one hand,
in Euclidean space, people incline to find the critical exponent regarding to $(p,q)$ which is usually related to the dimension of Euclidean space. But here,
we transfer to find
the ``sharp" volume growth condition, which may have no relation with the dimension of the manifold.
On the other hand, to our best knowledge, there seems much less known
when $p,q$ are negative even in Euclidean space. Some papers can only allow $p$ to be negative, but always require $q$ to be nonnegative, see \cite{CC15, CR20, FQS, Filip09, Filip11, FPS20,HB15, LiLi12}. Our paper will treat all the cases of $(p,q)$, and fill up
the gap when $q$ is negative, especially in the Euclidean space.

Our paper is greatly inspired by Grigor'yan and the first author's paper \cite{Grigoryan13}, which is very valid to deal with the semilinear differential inequality especially on manifolds.
Actually, the technique used in \cite{Grigoryan13} works only to the case of $p+q>m-1, p\geq0, q<m$.
Our technique presented here can be seen as a further refined version of \cite{Grigoryan13}, and
is quite powerful, and especially can be used to deal with the term $u^p|\nabla u|^q$ of negative power.
Hence, probably it is the first time to obtain the Liouville's theorem concerning $(p,q)$ of negative pair, and our technique here is quite novel in this respect.

As we see later, there are significant differences when $(p,q)$ takes different values. To express our classification more clearly, let us divide $\mathbb{R}^2$ into five parts (see Figure \ref{graph0})
\begin{align*}
&G_1= \{ (p, q)| p\geq 0, m-1-p<q<m \}, &&G_2= \{ (p, q)| q\geq m \}, \\
&G_3= \{(p, q)| p<0, m-1<q<m \}, 	&&G_4= \{(p, q)| p< m-1-q, q=m-1\},\\
&G_5= \{(p, q)| p= m-1-q, q\leq m-1\}, &&G_6=\{(p, q)| p< m-1-q, q< m-1\}.
\end{align*}

\begin{figure}[h]
\begin{tikzpicture}
\fill[blue, fill opacity=0.73] (-4, 2) rectangle (4, 3);
\fill[red, fill opacity=0.57] (-4, 1) rectangle (0, 2);
\fill[green, fill opacity=0.39] (0, 1)--(0, 2)--(4, 2)--(4, -3);
\fill[yellow, fill opacity=0.25] (0, 1)--(-4, 1)--(-4, -3)--(4, -3);
\draw[thick, ->] (-4, 0)--(4, 0) node[right] {$p$ };
\draw[thick, ->] (0, -3)--(0, 3) node[above] {$q$ };
\draw (-4, 2)--(4, 2);
\draw[dashed, very thick] (0, 1)--(4, -3);
\draw[dashdotted, very thick] (-4, 1)--(0, 1);
\node at (-3, 1.2) {\begin{scriptsize}$q=m-1$\end{scriptsize}};
\node at (-3.2, 2.2) {\begin{scriptsize}$q=m$\end{scriptsize}};
\node at (2.8, -.8) {\begin{scriptsize}$p+q=m-1$\end{scriptsize}};
\node at (5.2, 3.5-3) {$G_1$};
\node at (5.2, 3.5-2) {$G_3$};
\node at (5.2, 3.5-1) {$G_2$};
\node at (5.2, 3.5-4) {$G_6$};
\node at (5.2, 3.5-5) {$G_5$};
\node at (5.2, 3.5-6) {$G_4$};
\draw (4.55, 3.35-1) rectangle (4.85, 3.65-1);
\fill[blue, fill opacity=0.73] (4.55, 3.35-1) rectangle (4.85, 3.65-1);
\draw (4.55, 3.35-2) rectangle (4.85, 3.65-2);
\fill[red, fill opacity=0.57] (4.55, 3.35-2) rectangle (4.85, 3.65-2);
\draw (4.55, 3.35-3) rectangle (4.85, 3.65-3);
\fill[green, fill opacity=0.39] (4.55, 3.35-3) rectangle (4.85, 3.65-3);
\draw (4.55, 3.35-4) rectangle (4.85, 3.65-4);
\fill[yellow, fill opacity=0.25] (4.55, 3.35-4) rectangle (4.85, 3.65-4);
\draw[dashed, very thick] (4.3, 3.55-5)--(4.9, 3.55-5);
\draw[dashdotted, very thick] (4.25, 2.55-5)--(4.9, 2.55-5);
\end{tikzpicture}
\caption{$(p, q)$}
\label{graph0}
\end{figure}
Our main results are as follows:
\begin{theorem}\label{thm-non}
\rm{Let $M$ be a noncompact geodesically complete manifold.
\begin{enumerate}
\item[\textbf{(I)}]{ Assume $(p, q) \in G_1$. If
\begin{equation}\label{vol-1}
V(r) \lesssim r^{\frac{mp+q}{p+q-m+1}} \left( \ln r \right)^{\frac{m-1}{p+q-m+1}},\quad\mbox{for all large enough $r$},
\end{equation}
then (\ref{equ}) possesses no nontrivial positive solution.}

\item[\textbf{(II)}] {Assume $(p, q) \in G_2$. If
\begin{equation}\label{vol-2}
V(r)\lesssim r^{m}(\ln r)^{m-1},\quad\mbox{for all large enough $r$},
\end{equation}
then (\ref{equ}) possesses no nontrivial positive solution.}

\item[\textbf{(III)}]{ Assume $(p, q) \in G_3$. If
\begin{align}\label{vol-3}
V(r) \lesssim r^{\frac{q}{q-m+1}} \left( \ln r \right)^{\frac{m-1}{q-m+1}},\quad\mbox{for all large enough $r$},
\end{align}
then (\ref{equ}) possesses no nontrivial positive solution.}

\item[\textbf{(IV)}]{ Assume $(p, q) \in G_4$. For any given $\alpha>0$, if
\begin{align}\label{vol-4}
V(r) \lesssim r^{\alpha},\quad\mbox{for all large enough $r$},
\end{align}
then (\ref{equ}) possesses no nontrivial positive solution.}

\item[\textbf{(V)}]{Assume $(p, q) \in G_5$. For given $\kappa$ satisfying
$0<\kappa<\frac{\min\{m-1,1\}}{2e},$ 
if
\begin{equation}\label{vol-5}
V(r)\lesssim e^{\kappa r},\quad\mbox{for all large enough $r$},
\end{equation}
then (\ref{equ}) possesses no nontrivial positive solution.}

\item[\textbf{(VI)}]{ Assume $(p, q) \in G_6$. For given $\kappa$ satisfying $0<\kappa<\frac{m-1-q}{m-1-p-q}$, if
\begin{align}\label{vol-6}
V(r)\lesssim e^{\kappa r\ln r},\quad\mbox{for all large enough $r$},
\end{align}
then (\ref{equ}) possesses no nontrivial positive solution.}
\end{enumerate}}

\end{theorem}

\begin{remark}
\rm{ Concerning Theorem \ref{thm-non}, we have the following comments:
\begin{enumerate}
\item[(i)]{We emphasize that there is no any curvature assumption on manifold $M$. Theorem \ref{thm-non} \textbf{(I)} $\&$ \textbf{(III)} are understood in this way: when $(p, q)$ of problem (\ref{equ}) is fixed, the volume condition (\ref{vol-1}) and (\ref{vol-3})
show that how large the volume can grow to suffice the Liouville's results hold.}

\item[(ii)]{It is very interesting to see that the volume growth varies not uniformly with respect to $(p, q)$.
One of the reasons
to induce such phenomenon is that the interaction between $u^p$, $|\nabla u|^q$ and $\Delta_m u$. In $G_2, G_4, G_5$ and $G_6$, $(p,q)$ does not play any role in the volume growth condition; But in $G_3$, $p$ does not make any contribution to the volume growth. }

\item[(iii)]{In Theorem \ref{thm-non} \textbf{(II)}, volume condition (\ref{vol-2}) can be relaxed to a integral form
\begin{equation*}\label{vol-2-0}
\int^{+\infty} \left(\frac{r}{V(r)}\right)^{\frac{1}{m-1}} dr = \infty,
\end{equation*}
see (\ref{vol-m}) below.}


\item[(iv)]{Though $G_4, G_5$ are the boundaries of $G_6$, but the upper volume growths of $G_4, G_5, G_6$ behave quite different from each other.}

\item[(v)]{As one can see from the proof of Theorem \ref{thm-non}, Liouville's results are still valid to
$\Delta_m u+lu^p|\nabla u|^q\leq0$ on $M$, where $l$ is some positive constant.}
\end{enumerate}
}
\end{remark}

By letting $p=q=0$, and replacing $u$ by $cu$, it is easy to derive the following from Theorem \ref{thm-non} \textbf{(VI)}.
\begin{corollary}\label{cor-1}
\rm{Let $M$ be a noncompact geodesically complete manifold.
For given $\kappa$ satisfying $0<\kappa<1$, if
\begin{align*}\label{vol-3-4-2}
V(r)\lesssim e^{\kappa r\ln r},\quad\mbox{for all large enough $r$},
\end{align*}
then for any constant $l>0$, the following differential inequality
\begin{equation}\label{msuper}
\Delta_m u+l\leq0,\quad\mbox{on $M$}
\end{equation}
admits no positive solution.}
\end{corollary}

\begin{remark}
\rm{The above corollary
can be seen as a generalization of Liouville's results concerning the so-called $m\text{-}$superharmonic function (namely,  we say $u$ in (\ref{msuper}) of $l=0$ is
  $m\text{-}$superharmonic), see Cheng-Yau's paper \cite{ChengYau} for $m=2$, and Holopainen's paper \cite{Holo99} for general $m>1$. For more Liouville's results in this area, please see Serrin-Zou's paper \cite{SZ02}. }
\end{remark}

Now let us give some motivation for using volume condition. Though the volume of geodesic ball is a very simple geometric quantity of manifold, but it plays a very important role when one studies the geometric and
probabilistic properties, and also estimate of heat kernel of Laplace-Beltrami operator on manifolds, please see a recent nice survey \cite{Grigoryan20}.

Among these properties,
the interaction between volume growth and nonexistence of solution to elliptic differential inequalities has been known for a long time.
A pioneering theorem of Cheng-Yau in the paper \cite{ChengYau}
says that if there exists a sequence $r_k\to\infty$ such that for all $k>0$
\begin{align*}
V(r_k)\lesssim r_k^2,
\end{align*}
then $M$ is parabolic. Here we call that a manifold $M$ is parabolic, if any positive superharmonic function on $M$ is constant.

Later, Grigor'yan, Karp, Varopoulos \cite{Grigoryan83, Karp82, Va83} independently studied parabolicity of manifold, and generalized Cheng-Yau's result to an integral version, namely, if
\begin{align}\label{vol-pra}
\int^{\infty}\frac{r}{V(r)}dr=\infty,
\end{align}
then $M$ is parabolic. Here $\int^{\infty}$ means that we take the integral near infinity.
It is worth to point out that the condition (\ref{vol-pra}) is sharp but not a necessary condition for parabolicity of $M$, see \cite{Grigoryan99, Grigoryan20}.

In \cite{Holo99}, Holopainen investigated  $m\text{-}$superharmonic function, and proved that, if
\begin{equation}\label{vol-m}
\int^{\infty}\left(\frac{r}{V(r)}\right)^{\frac{1}{m-1}}dr=\infty,
\end{equation}
then any nonnegative solution to $\Delta_m u\leq0$ is identical constant. In particular, (\ref{vol-m}) is satisfied
if (\ref{vol-2}) holds for all large enough $r$.

In \cite{Grigoryan13}, Grigor'yan and the first author studied (\ref{equ}) under the case of $m=2$, $q=0$ and $p>1$, which is
\begin{align}\label{equ-sem}
\Delta u+u^p\leq0,\quad\mbox{on $M$}.
\end{align}
They obtained that if
\begin{align}\label{vol-sem}
V(r)\lesssim r^{\frac{2p}{p-1}}(\ln r)^{\frac{1}{p-1}},\quad\mbox{for all large enough $r$},
\end{align}
then any nonnegative solution to (\ref{equ-sem}) is identical zero.
The exponents $\frac{2p}{p-1}$ and $\frac{1}{p-1}$ in (\ref{vol-sem})
are sharp, and
cannot be relaxed here, namely, if the exponent $\frac{1}{p-1}$ is relaxed a little by $\frac{1}{p-1}+\epsilon$ for any positive $\epsilon$
close to zero, there exists a manifold satisfying (\ref{vol-sem}) but admits a positive solution to (\ref{equ-sem}).
The above result can be also reformulated equivalently as follows:
if for given $\alpha>2$,
\begin{align*}
V(r)\lesssim r^{\alpha}(\ln r)^{\frac{\alpha-2}{2}},\quad\mbox{for all large enough $r$},
\end{align*}
then for $1<p\leq \frac{\alpha}{\alpha-2}$, any nonnegative solution to (\ref{equ-sem})
is identical zero.

Concerning quasilinear differential inequalities, the Liouville's result was obtained by the first author in \cite{Sun15} under the case of
$m>1, p>m-1, q=0$: if
\begin{equation}
V(r)\lesssim r^{\frac{mp}{p-m+1}}(\ln r)^{\frac{m-1}{p-m+1}},\quad\mbox{for all large enough $r$},
\end{equation}
then
\begin{equation*}
\Delta_m u+u^p\leq0,\quad\mbox{on $M$}
\end{equation*}
admits no nonnegative solution except zero. Here the exponents $\frac{mp}{p-m+1}$ and $\frac{m-1}{p-m+1}$ are also sharp, and cannot be relaxed.

Besides the above mentioned literature, many analogous results using volume growth to derive the Liouville's theorems are obtained when studying more general differential inequalities, see \cite{MMP15, Sun-jmaa, Sun15, Wang-Xiao, XWS, Xu19}.
For example, by using the perturbation of fundamental solution to
$\Delta_m$, Wang-Xiao presented a constructive approach to obtain the quantitative positive solution so that they can show the sharpness
of the volume growth (cf. \cite{Wang-Xiao}).
Mastrolia-Monticelli-Punzo investigated the quasilinear differential inequality with potential term, and they showed that the potential function
gives a direct influence on the volume growth, namely, if the potential decays fast enough at infinity,
then the sharpness of critical exponent for the log-term will expire (cf. \cite{MMP15}).

Recently in this direction, volume growth condition is also used to derive the nonexistence and existence of global positive solution to parabolic equation on manifolds, see
\cite{GSXX, MMP17, SX20}.

Noting in the above mentioned literature, there is only volume condition without imposing any other assumptions on manifolds. Now for our
convenience, let us introduce the following additional restrictions on manifolds:
\begin{enumerate}
\item[(VD)]{The volume doubling condition: for all $x\in M$ and $r>0$
\begin{align*}
\mu(B(x, 2r))\lesssim \mu(B(x,r)).
\end{align*} }
\item[(PI)]{The Poincar\'{e} inequality: for any ball $B(x,r)$ and $f\in C^2(B(x,r))$
\begin{align*}
\int_{B(x,r)}|f-f_B|^2d\mu \lesssim r^2\int_{B(x,r)}|\nabla f|^2d\mu,
\end{align*}
where $f_B$ stands for the mean value of $f$ on $B(x,r)$.}
\end{enumerate}

Recently in the paper \cite{Sun20}, Grigor'yan, Verbitsky and the first author proved that on the manifold where the above mentioned conditions $\text{(VD)}$ and $\text{(PI)}$ are both satisfied, then (\ref{equ-sem}) possesses a $C^2$ positive solution
if and only if
\begin{align}\label{vol-int-1}
\int^{\infty}\frac{r^{2p-1}}{V(r)^{p-1}}dr<\infty,
\end{align}
or equivalently
\begin{align}\label{vol-int-2}
\int^{\infty}\left[\int_r^{\infty}\frac{tdt}{V(t)}\right]^{p-1}rdr<\infty,
\end{align}
holds. The equivalence of volume conditions (\ref{vol-int-1}) and (\ref{vol-int-2})
can be derived by applying the Hardy's type inequality.
Though (\ref{vol-int-1}) and (\ref{vol-int-2}) are more delicate, and
can be considered as an integral version of (\ref{vol-sem}) (by replacing ``$<$" with ``=" in (\ref{vol-int-1}) and (\ref{vol-int-2})). However, besides (\ref{vol-int-1}) and (\ref{vol-int-2}),
more additional assumptions
on manifold are needed. Hence it is very natural to ask whether it is possible to remove the assumptions of
$\text{(VD)}$ and $\text{(PI)}$ to verify the following conjecture.
\begin{conjecture}
\rm{Let $M$ be a noncompact geodesically complete manifold.
If either
(\ref{vol-int-1}) or (\ref{vol-int-2}) is satisfied on $M$, then any nonnegative solution to (\ref{equ-sem}) is identical
zero.}
\end{conjecture}

Besides these connections with existence and nonexistence of solution to partial differential equation and inequalities, volume
condition
can be also used to obtain the information on the
stochastic completeness of manifold, and heat kernel of $\Delta$, see \cite{BCG, CG, Grigoryan99, Grigoryan20}. Recall that $M$ is stochastic complete, if the life time of Brownian motion on $M$
is almost equal to $\infty$ almost surely. It is well known that if
\begin{align}\label{vol-sto}
\int^\infty\frac{rdr}{\ln V(r)}=\infty.
\end{align}
then $M$ is stochastic complete.
Condition (\ref{vol-sto}) is also sharp but not necessary for the stochastic completeness of the manifold (cf. \cite[Propostion 3.2]{Grigoryan99}).
Another well known result is that if $V(r)\lesssim r^{\alpha}$ hold for all large enough $r$, then the heat kernel $p_t(o, o)\gtrsim\frac{1}{(t\ln t)^{\frac{\alpha}{2}}}$ holds for large enough $t$, see \cite{CG}.

Apart from the motivation of volume growth, we are also inspired by the literature dealing with (\ref{equ}) in Euclidean space. In \cite{MP01},
Mitidieri and Pohozaev proved that when $M=\mathbb{R}^n$, $p>0, q\geq0$,
and $p+q>m-1$, $n>m$. If
\begin{align}\label{MP-pq}
p(n-m)+q(n-1)\leq n(m-1),
\end{align}
then (\ref{equ})
admits no positive solution except constants. Noting that (\ref{MP-pq}) is equivalent to
$$n\leq \frac{mp+q}{p+q-m+1},$$
which is also recovered by Theorem \ref{thm-non} \textbf{(I)}.

Later, many authors generalized Mitidieri-Pohozaev's results to more general differential inequalities, including
coercive and anticoercive elliptic inequalities or system, both of the $m\text{-}$laplacian and of the mean curvature type.
Since there are huge numbers of the papers in this area and its generalization, we apologize we cannot list all the bibliography exhaustively here, please see
\cite{BGV, Filip09, Filip11, Lions85, MP99, MP01, SZ02} and references therein.

Though there are many papers devoted to study the elliptic differential equations or differential inequalities of type (\ref{equ}) in bounded domain and unbounded domain, but very surprisingly, there are few papers to deal with the problem of (\ref{equ}) and its equation version with negative $(p,q)$. Among of these literature, $p$ is usually allowed to be negative usually under some restriction, but $q$ is always nonnegative, see \cite{CC15, CR20, FQS, Filip09, Filip11, FPS20,HB15, LiLi12}.

As far as we know, even in the
Euclidean settings, probably this is the first paper
to deal with case when $p, q$ are both allowed to be negative.

Recently in \cite{BGV}, Bidaut-V\'{e}ron, Garc\'{i}a-Huidobro, and V\'{e}ron studied local and global properties of positive solutions
to
\begin{equation}\label{eq-eq}
\Delta u+u^p|\nabla u|^q=0,
\end{equation}
in a domain $\Omega\subset\mathbb{R}^n$, and espcecially, they proved that
if $p, q\geq0$ and
\begin{equation*}
p(n-2)+q(n-1)<n,
\end{equation*}
then (\ref{eq-eq}) admits no positive solution except constants. Especially when $\Omega=\mathbb{R}^n$, there exists nonradial positive
solutions of (\ref{eq-eq}) if and only if $p\geq0, 0\leq q<1$ and
\begin{equation*}
p(n-2)+q(n-1)\geq n+\frac{2-q}{1-q}.
\end{equation*}

In \cite{FPS20}, Filippucci, Pucci and Souplet restudied (\ref{eq-eq}) in $\mathbb{R}^n$, they generalized \cite{BGV}'s result to the case of
$q>2, p+q>1$(noting here $p$ can be a little ``negative"), and obtained that the only positive bounded classical solution to (\ref{eq-eq}) is constant.

Motivated by above literature, we also show that our volume growth conditions
in Theorem \ref{thm-non} are sharp for in the most cases, roughly speaking, if the volume conditions in Theorem \ref{thm-non} are relaxed, then there exist manifolds which admits a nontrivial positive solution to (\ref{equ}).
\begin{theorem}\label{thm-ex}
\rm{For any given $\epsilon>0$.
\begin{enumerate}
\item[\textbf{(A)}]{ Assume $(p, q) \in G_1$. Then
there exists a noncompact geodesically complete manifold $M$ satisfying
\begin{align}\label{vol-1-1}
V(r) \lesssim r^{\frac{mp+q}{p+q-m+1}} \left( \ln r \right)^{\frac{m-1}{p+q-m+1}+\epsilon},\quad\mbox{for all large enough $r$},
\end{align}
such that (\ref{equ}) possesses a nontrivial positive solution.}

\item[\textbf{(B)}] {Assume $(p, q) \in G_2$.
Then there exists a noncompact geodesically complete manifold $M$ satisfying
\begin{align}\label{vol-2-1}
V(r) \lesssim r^{m} \left( \ln r \right)^{m-1+\epsilon},\quad\mbox{for all large enough $r$},
\end{align}
such that (\ref{equ}) possesses a nontrivial positive solution.}

\item[\textbf{(C)}]{ Assume $(p, q) \in G_3$. Then
there exists a noncompact geodesically complete manifold $M$ satisfying
\begin{align}\label{vol-3-1}
V(r) \lesssim r^{\frac{q}{q-m+1}} \left( \ln r \right)^{\frac{1}{q-m+1} +\epsilon},\quad\mbox{for all large enough $r$},
\end{align}
such that (\ref{equ}) possesses a nontrivial positive solution.}

\item[\textbf{(D)}]{Assume $(p,q) \in G_4$. Given $\lambda>0$,
then there exist a noncompact geodesically complete manifold $M$ satisfying
\begin{equation}\label{vol-4-1}
V(r)\lesssim e^{\lambda r}, \quad\mbox{for all large enough $r$},
\end{equation}
such that (\ref{equ}) possesses a
nontrivial positive solution.}

\item[\textbf{(E)}] {Assume $(p,q) \in G_5$. Then there exist some positive constant $\iota$ and a noncompact geodesically complete manifold $M$ satisfying
\begin{equation}\label{vol-5-1}
V(r)\lesssim e^{\iota r}, \quad\mbox{for all large enough $r$},
\end{equation}
such that (\ref{equ}) possesses a
nontrivial positive solution.}

\item[\textbf{(F)}] {Assume $(p,q) \in G_6$. Given $\lambda>0$, then there exist some positive constant $\iota>2(m-1-q)+1$ and a noncompact geodesically complete manifold $M$ satisfying
\begin{equation}\label{vol-6-1}
V(r)\lesssim e^{\lambda r^{\iota}\ln r}, \quad\mbox{for all large enough $r$},
\end{equation}
such that (\ref{equ}) possesses a
nontrivial positive solution.}
\end{enumerate}}
\end{theorem}

\begin{remark}
\rm{ We have the following comments
\begin{enumerate}
\item[1.]{In Theorem \ref{thm-ex} \textbf{(D)}, we can see that when $(p, q)\in G_4$, if the volume has exponential growth which
means that (\ref{vol-4}) fails, then the Liouville's result may not work.}

\item[2.]{In Theorem \ref{thm-ex} \textbf{(E)}, we do not show how small of $\iota$ in (\ref{vol-5-1}) for general manifold to suffice the existence of
a nontrivial positive solution to
(\ref{equ}). But in the proof of Theorem \ref{thm-ex}, we show that for some special model manifold, if $\iota$ is big enough (see (\ref{5-3-iota}) and (\ref{5-iota})), then (\ref{equ}) admits a nontrivial positive solution. }

\item[3.]{In Theorem \ref{thm-ex} \textbf{(F)}, we do not know whether $2(m-1-q)+1$ is sharp for $\iota$, but it also shows that
the volume growth $e^{\kappa r\ln r}$ in (\ref{vol-6}) is sharp in profile, which can not been replaced by $e^{\kappa r^{1+\epsilon}\ln r}$ for any $\epsilon>0$.}
\end{enumerate}}
\end{remark}

\emph{\textbf{Notations.} In the above and below, the letters $C,C^{\prime },C_{0},C_{1}, c_0, c_1...$ denote positive
constants whose values are unimportant and may vary at different occurrences.
$U\lesssim V$ stands for $U\le c V$ for a constant $c>0$; $U\asymp V$ means both $U\lesssim V$ and $V\lesssim U$.}

\section{Nonexistence}\label{sec-non}

Denote by $%
W_{loc}^{1,m }\left( M\right) $ the space of functions $f\in L_{loc}^{m}\left(
M\right) $ whose weak gradient $\nabla f$ is also in $L_{loc}^{m}\left(
M\right) .$ Denote by $W_{c}^{1,m}\left( M\right) $ the subspace of $%
W_{loc}^{1,m }\left( M\right) $ of functions with compact support.

Solutions of (\ref{equ}) are understood in a weak sense.
\begin{definition}
$u$ is called a weak positive solution to (\ref{equ}), if $u$ is a positive function from $W_{loc}^{1,m}(M)$, and for any nonnegative function $ \psi\in W_{c}^{1,m}(M)$, the following holds:
\begin{equation}\label{sol-def}
-\int_{M}|\nabla u|^{m-2}(\nabla u, \nabla \psi)d\mu+\int_{M}u^p|\nabla u|^q\psi d\mu\leq0.
\end{equation}
\end{definition}

\begin{lemma}\label{lem1}
\rm{Assume $p+q\neq m-1$. If $u$ is a nontrivial positive solution to (\ref{equ}), then there exists a positive pair $(a, b)$ such that
for any $0\leq\varphi\leq1$ with $\varphi\in W_{c}^{1,m}(M)$, the following estimates hold:
\begin{align}\label{2-5}
\int_M u^{p-a} |\nabla u|^q \varphi^b d\mu
&\lesssim  (2b)^{\frac{mp+q+a(q-m)}{p+q-a}}a^{-\frac{p(m-1)+a(q-m+1)}{p+q-a}} \left( \int_{\text{supp}|\nabla\varphi|} u^{p-a} |\nabla u|^q \varphi^b d\mu \right)^{\frac{m-1-a}{p+q-a}} \notag\\
&\quad \times \left( \int_{M} |\nabla \varphi|^{\frac{mp+q+a(q-m)}{p+q-m+1}} d\mu \right)^{\frac{p+q-m+1}{p+q-a}},
\end{align}
and
\begin{align}\label{est-1}
\int_M u^{p-a} |\nabla u|^q \varphi^b d\mu
\lesssim (2b)^{\frac{mp+q+a(q-m)}{p+q-m+1}} a^{-\frac{p(m-1)+a(q-m+1)}{p+q-m+1}} \int_{M} |\nabla \varphi|^{\frac{mp+q+a(q-m)}{p+q-m+1}} d\mu,
\end{align}
where $a, b$ satisfy
\begin{equation}\label{cond-ab}
\left\{
\begin{array}{ll}
\frac{mp+q+a(q-m)}{(m-1)p+a(q-m+1)}>1, \\
\frac{p+q-a}{m-1-a}>1,\\
b>\frac{mp+q+a(q-m)}{p+q-m+1}.
\end{array}
\right.
\end{equation}
}
\end{lemma}
\begin{proof}
Take  $\varphi\in W_{c}^{1,m}(M)$ such that $0\leq\varphi\leq1$.
Without loss of generality, let us assume that $u^{-1}\in L_{loc}^{\infty}(M)$, otherwise we replace $u$ by $u+\epsilon$ for any $\epsilon>0$, and
at last let $\epsilon$ goes to zero.
Substituting $\psi=u^{-a}\varphi^b$ in (\ref{sol-def}), we obtain
\begin{align}\label{2-1}
\int_M u^{p-a} |\nabla u|^q \varphi^b d\mu
+a \int_M u^{-a-1} |\nabla u|^m \varphi^b d\mu
\leq b\int_M u^{-a} \varphi^{b-1} |\nabla u|^{m-2}(\nabla u, \nabla \varphi) d\mu,
\end{align}
where $a, b>0$ are to be chosen later.

Let
\begin{align}\label{def-st}
s=\frac{mp+q+a(q-m)}{(m-1)p+a(q-m+1)}, \quad
t=\frac{mp+q+a(q-m)}{p+q-a}.
\end{align}
By choosing $a$ to let $s, t>1$, and applying the following Young's inequality
$$XY\leq X^s+Y^t,\quad\mbox{for $X, Y>0$},$$
we obtain
\begin{align}\label{2-2}
&b \int_M u^{-a} \varphi^{b-1}|\nabla u|^{m-2} (\nabla u, \nabla \varphi) d\mu \nonumber\\
&\leq
\int_M
\left[\left(\frac{a}{2}\right)^{\frac{1}{s}} u^{-\frac{a+1}{s}} |\nabla u|^{\frac{m}{s}} \varphi^{\frac{b}{s}} \right]
\left[ b\left(\frac{a}{2}\right)^{-\frac{1}{s}} u^{-a+\frac{a+1}{s}} |\nabla u|^{m-1-\frac{m}{s}} \varphi^{b-1-\frac{b}{s}} |\nabla \varphi| \right] d\mu \notag\\
&\leq \frac{a}{2} \int_M u^{-a-1} |\nabla u|^m \varphi^b d\mu \notag\\
&\quad+ b^{t}2^{\frac{t}{s}} a^{-\frac{t}{s}} \int_M u^{-a t+\frac{t}{s}(a+1)} |\nabla u|^{t(m-1)-\frac{mt}{s}} \varphi^{t(b-1)-\frac{bt}{s} } |\nabla \varphi|^t d\mu \notag\\
&= \frac{a}{2} \int_M u^{-a-1} |\nabla u|^m \varphi^b d\mu + b^{t}2^{t-1} a^{1-t} \int_M u^{-a+t-1} |\nabla u|^{m-t} \varphi^{b-t} |\nabla \varphi|^t d\mu.
\end{align}
Substituting (\ref{2-2}) into (\ref{2-1}), we have
\begin{align*}
&\int_M u^{p-a} |\nabla u|^q \varphi^b d\mu
+ \frac{a}{2} \int_M u^{-a-1} |\nabla u|^m \varphi^b d\mu \notag\\
&\lesssim (2b)^ta^{1-t} \int_M u^{-a+t-1} |\nabla u|^{m-t} \varphi^{b-t} |\nabla \varphi|^t d\mu.
\end{align*}
Hence
\begin{align}\label{2-3}
&\int_M u^{p-a} |\nabla u|^q \varphi^b d\mu \lesssim (2b)^ta^{1-t} \int_M u^{-a+t-1} |\nabla u|^{m-t} \varphi^{b-t} |\nabla \varphi|^t d\mu.
\end{align}
Define
\begin{align}\label{def-kl}
\gamma = \frac{p+q-a}{m-1-a}, \quad
\rho = \frac{p+q-a}{p+q-m+1}.
\end{align}
Letting $a$ be chosen further such that $\gamma, \rho>1$, and applying H\"older's inequality to the second integral of (\ref{2-3}), we obtain
\begin{align*}
& \int_M u^{-a+t-1} |\nabla u|^{m-t} \varphi^{b-t} |\nabla \varphi|^t d\mu \notag\\
&= \int_M \left( u^{-a+t-1} |\nabla u|^{m-t} \varphi^{\frac{b}{\gamma}} \right)
\left( \varphi^{b-t-\frac{b}{\gamma}} |\nabla \varphi|^t \right) d\mu \notag\\
&\leq
\left( \int_{\text{supp}|\nabla\varphi|} u^{p-a} |\nabla u|^q \varphi^b d\mu \right)^{\frac{1}{\gamma}}
\left( \int_{M} \varphi^{b-t\rho} |\nabla \varphi|^{t\rho} d\mu \right)^{\frac{1}{\rho}}.
\end{align*}
Letting $b$ be chosen large enough such that $b\geq t\rho$, and
noting $0\leq \varphi \leq 1$, we derive
\begin{align}\label{2-4}
& \int_M u^{-a+t-1} |\nabla u|^{m-t} \varphi^{b-t} |\nabla \varphi|^t d\mu \notag\\
&\leq
\left( \int_{\text{supp}|\nabla\varphi|} u^{p-a} |\nabla u|^q \varphi^b d\mu \right)^{\frac{1}{\gamma}}
\left( \int_{M} |\nabla \varphi|^{t\rho} d\mu \right)^{\frac{1}{\rho}}.
\end{align}
Combining (\ref{2-4}) with (\ref{2-3}), we obtain
\begin{align*}
\int_M u^{p-a} |\nabla u|^q \varphi^b d\mu
\lesssim (2b)^ta^{1-t} \left( \int_{\text{supp}|\nabla\varphi|} u^{p-a} |\nabla u|^q \varphi^b d\mu \right)^{\frac{1}{\gamma}}
\left( \int_{M} |\nabla \varphi|^{t\rho} d\mu \right)^{\frac{1}{\rho}},
\end{align*}
namely,
\begin{align*}
\int_M u^{p-a} |\nabla u|^q \varphi^b d\mu
&\lesssim (2b)^{\frac{mp+q+a(q-m)}{p+q-a}}a^{-\frac{p(m-1)+a(q-m+1)}{p+q-a}} \left( \int_{\text{supp}|\nabla\varphi|} u^{p-a} |\nabla u|^q \varphi^b d\mu \right)^{\frac{m-1-a}{p+q-a}} \notag\\
&\quad\times \left( \int_{M} |\nabla \varphi|^{\frac{mp+q+a(q-m)}{p+q-m+1}} d\mu \right)^{\frac{p+q-m+1}{p+q-a}},
\end{align*}
which yields (\ref{2-5}).

Since $\int_M u^{p-a} |\nabla u|^q \varphi^b d\mu$ is finite due to the compactness of the support of $\varphi$, it follows that
\begin{align*}
&\left( \int_M u^{p-a} |\nabla u|^q \varphi^b d\mu \right)^{\frac{p+q-m+1}{p+q-a}} \notag\\
&\lesssim (2b)^{\frac{mp+q+a(q-m)}{p+q-a}}a^{-\frac{p(m-1)+a(q-m+1)}{p+q-a}} \left( \int_{M} |\nabla \varphi|^{\frac{mp+q+a(q-m)}{p+q-m+1}} d\mu \right)^{\frac{p+q-m+1}{p+q-a}},
\end{align*}
which implies (\ref{est-1}).

Besides, the existence of pair $(a, b)$ is easy to obtain from (\ref{cond-ab}). Hence, we complete the proof.
\end{proof}

\begin{remark}\label{rem-3}
In Lemma \ref{lem1}, since $b$ is only need to be chosen large enough, it suffices to verify that
such $a$ exists.
For our convenience, let us divide $\mathbb{R}^2\setminus\{p+q=m-1\}$ into four different parts $K_1, K_2, K_3, K_4$
(see Figure \ref{graph}):
\begin{align*}
&K_1 = \{(p, q)| p < m-1-q, q\leq m-1 \}, 		&& K_2 = \{(p, q)| p\geq 0, m-1-p<q\leq m-1 \} \\
&K_3 = \{(p, q)| p>m-1-q, q>m-1 \}, 			&& K_4 = \{ (p, q)|p < 0, m-1<q<m-1-p \}
\end{align*}
\end{remark}

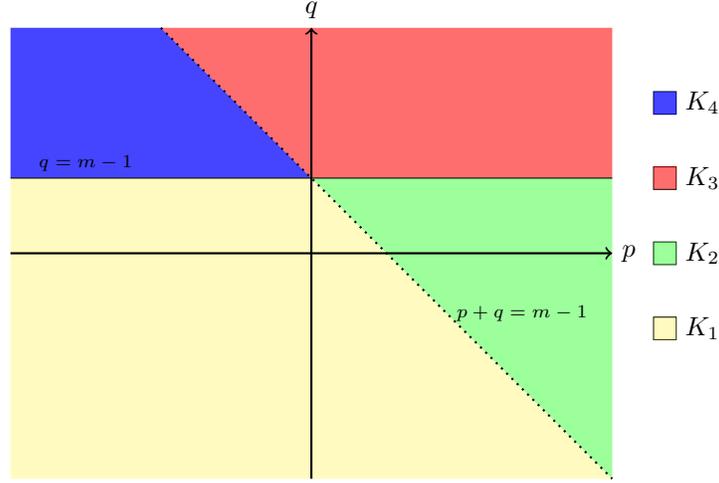
\begin{figure}[h]
\begin{tikzpicture}
\fill[blue, fill opacity=0.73] (-4, 1)--(-4, 3)--(-2, 3)--(0, 1);
\fill[red, fill opacity=0.57] (-2, 3)--(0, 1)--(4, 1)--(4, 3);
\fill[green, fill opacity=0.39] (0, 1)--(4, 1)--(4, -3);
\fill[yellow, fill opacity=0.25] (0, 1)--(-4, 1)--(-4, -3)--(4, -3);
\draw[thick, ->] (-4, 0)--(4, 0) node[right] {$p$ };
\draw[thick, ->] (0, -3)--(0, 3) node[above] {$q$ };
\draw[dotted, thick] (-2, 3)--(4, -3);
\draw (-4, 1)--(4, 1);
\node at (5.2, 3-1) {$K_4$};
\node at (5.2, 3-2) {$K_3$};
\node at (5.2, 3-3) {$K_2$};
\node at (5.2, 3-4) {$K_1$};
\draw (4.55, 2.85-1) rectangle (4.85, 3.15-1);
\fill[blue, fill opacity=0.73] (4.55, 2.85-1) rectangle (4.85, 3.15-1);
\draw (4.55, 2.85-2) rectangle (4.85, 3.15-2);
\fill[red, fill opacity=0.57] (4.55, 2.85-2) rectangle (4.85, 3.15-2);
\draw (4.55, 2.85-3) rectangle (4.85, 3.15-3);
\fill[green, fill opacity=0.39] (4.55, 2.85-3) rectangle (4.85, 3.15-3);
\draw (4.55, 2.85-4) rectangle (4.85, 3.15-4);
\fill[yellow, fill opacity=0.25] (4.55, 2.85-4) rectangle (4.85, 3.15-4);
\node at (-3, 1.2) {\begin{scriptsize}$q=m-1$\end{scriptsize}};
\node at (2.8, -.8) {\begin{scriptsize}$p+q=m-1$\end{scriptsize}};
\end{tikzpicture}
\caption{$K_1, K_2, K_3, K_4 $}
\label{graph}
\end{figure}

Based on location of $(p,q)$, we give the following admissible choice of $a$:
\begin{center}
\begin{align*}
\text{If }(p, q) \in \left\{
\begin{array}{l}
K_1, \\K_2, \\K_3, \\K_4,
\end{array}
\right.
\text{we can choose } a \text{ such that}
\left\{
\begin{array}{l}
a>m-1 					. \\
a<m-1 					. \\
\frac{p(1-m)}{q-m+1}<a<m-1		. \\
m-1<a<\frac{p(1-m)}{q-m+1}		. \\
\end{array}
\right.
\end{align*}
\end{center}

Now we are ready to give the proof of Theorem \ref{thm-non}.
\begin{proof}[Proof of Theroem \ref{thm-non} \textbf{(I)}]
Let $h$ be a smooth function defined in $[0,\infty)$ such that $0\leq h\leq1$ and
\begin{align}\label{def-h}
\begin{array}{l}
h(t)=1, \;t\in\left[ 0, 1 \right];
\quad h(t)=0,\; t\in \left[2, +\infty \right),\quad |h^{\prime}(t)|\leq \rho,\; t\in[0,\infty).
\end{array}
\end{align}
where $\rho>1$.

Let us define
\begin{align*}
\eta _k(x) = h\left(\frac{r(x)}{2^{k}}\right),
\end{align*}
where $r(x)=d(x,o)$.

Consider a sequence of functions $\{\varphi_{i}(x)\}_{i\in\mathbb{N}}$ by
\begin{align}\label{test-i}
\varphi_{i}(x)=\frac{1}{i}\sum\limits_{k=i+1}^{2i} \eta_{k}(x),
\end{align}
It is easy to see that $\varphi_i(x)=1$ for $x\in B_{2^{i+1}}$, and $\varphi_i(x)=0$ for $x\in B_{2^{2i+1}}^c$.
By the disjointness of the support of $\nabla\eta_k$, we have
\begin{align}\label{3-1}
|\nabla \varphi_{i}(x)|^\theta =
\frac{1}{i^\theta}\sum\limits_{k=i+1}^{2i} |\nabla\eta_{k}(x)|^\theta \lesssim
\frac{1}{i^\theta}\sum\limits_{k=i+1}^{2i} \frac{1}{2^{k\theta}}
\chi_{\{ 2^k \leq r(\cdot) \leq 2^{k+1}\}}(x).
\end{align}
Substituting $\varphi=\varphi_i$ in (\ref{est-1}), we arrive
\begin{align}\label{3-2}
& \int_{B_{2^{i+1}}} u^{p-a} |\nabla u|^q d\mu \notag\\
& \lesssim (2b)^{\frac{mp+q+a(q-m)}{p+q-a}}a^{-\frac{p(m-1)+a(q-m+1)}{p+q-m+1}}\int_{\text{supp}|\nabla\varphi_i|}|\nabla\varphi_i|^{\frac{mp+q+a(q-m)}{p+q-m+1}}d\mu
\nonumber\\
& \lesssim (2b)^{\frac{mp+q+a(q-m)}{p+q-a}} \frac{a^{-\frac{p(m-1)+a(q-m+1)}{p+q-m+1}}}{i^{\frac{mp+q+a(q-m)}{p+q-m+1}}}
\sum\limits_{k=i+1}^{2i} \int_{B_{2^{k+1}} \backslash B_{2^{k}} }
2^{-k\frac{mp+q+a(q-m)}{p+q-m+1}} d\mu \notag\\
& \lesssim (2b)^{\frac{mp+q+a(q-m)}{p+q-a}}\frac{a^{-\frac{p(m-1)+a(q-m+1)}{p+q-m+1}}}{i^{\frac{mp+q+a(q-m)}{p+q-m+1}}}
\sum\limits_{k=i+1}^{2i} V(2^{k+1})
2^{-k\frac{mp+q+a(q-m)}{p+q-m+1}}.
\end{align}
Since $G_1\subset K_2\cup K_3$, by Remark \ref{rem-3}, let us take
\begin{align*}
a=\frac{1}{i}.
\end{align*}
and $b$ be some large fixed constant.

Recalling (\ref{vol-1}), and noting that $p+q>m-1$ and $q<m$, we obtain
\begin{align}\label{3-3}
\int_{B_{2^{i+1}}} u^{p-\frac{1}{i}} |\nabla u|^q d\mu & \lesssim i^{-1-\frac{m-1-1/i}{p+q-m+1}}
\sum\limits_{k=i+1}^{2i}
2^{\frac{k(m-q)}{i(p+q-m+1)}} k^{\frac{m-1}{p+q-m+1}} \notag\\
& \lesssim i^{-1+\frac{1/i}{p+q-m+1}}
\sum\limits_{k=i+1}^{2i}
2^{\frac{k(m-q)}{i(p+q-m+1)}} \nonumber\\
& \lesssim i^{\frac{1/i}{p+q-m+1}}.
\end{align}
Letting $i\rightarrow +\infty$ in (\ref{3-3}), by Fatou's lemma, we obtain
\begin{align}\label{3-bdd}
\int_{M} u^{p} |\nabla u|^q d\mu<\infty.
\end{align}
Substituting $\varphi=\varphi_{i}$ into (\ref{2-5}),
and repeating the same procedures, we obtain
\begin{align}\label{3-4}
\int_{B_{2^{i+1}}} u^{p-\frac{1}{i}} |\nabla u|^q d\mu \lesssim \left( \int_{M\backslash B_{2^{i+1}}} u^{p-\frac{1}{i}} |\nabla u|^q d\mu \right)^{\frac{m-1-1/i}{p+q-1/i}}.
\end{align}
Combining with (\ref{3-bdd}), and by letting $i\rightarrow +\infty$ again in (\ref{3-4}), we have
\begin{align}\label{3-5}
\int_M u^{p} |\nabla u|^q d\mu = 0,
\end{align}
which induces a contradiction with that $u$ is nontrivial.
Hence, we complete of the proof of Theorem \ref{thm-non} \textbf{(I)}.
\end{proof}

By the result of Holopainen \cite[Proposition 1.7]{Holo99}, it is easy to derive that under (\ref{vol-2}), then there exists no
nontrivial positive solution to (\ref{equ}). However, to keep the consistency of the paper, we present another proof which is totally different from
Holopainen's.

\begin{proof}[Proof of Theorem \ref{thm-non} \textbf{(II)}]

We divide the proof into three cases:

\begin{itemize}
\item[(II-1)] $(p, q) \in \{p >m-1-q, q\geq m \}$,
\item[(II-2)] $(p, q) \in \{ p =m-1-q, q\geq m \} $
\item[(II-3)] $(p, q) \in \{ p<m-1-q, q\geq m \} $
\end{itemize}

\textbf{In case (II-1)}, since now $(p, q)\in K_2$, let
\begin{equation}\label{a-2-1}
a=m-1-\frac{1}{i}.
\end{equation}
and $b$ be some large fixed constant.

Substituting $\varphi=\varphi_i$ from (\ref{test-i}) into (\ref{est-1}), and using the same technique as in (\ref{3-2}), we obtain
\begin{align*}
\int_{B_{2^{i+1}}} u^{p-a} |\nabla u|^q d\mu
&\lesssim (2b)^{\frac{mp+q+a(q-m)}{p+q-a}}\frac{a^{-\frac{p(m-1)+a(q-m+1)}{p+q-m+1}}}{i^{\frac{mp+q+a(q-m)}{p+q-m+1}}}\nonumber\\
&\quad \times
\sum_{k=i+1}^{2i}\int_{B_{2^{k+1}}\setminus B_{2^k}}2^{-k\frac{mp+q+a(q-m)}{p+q-m+1}}d\mu,
\end{align*}
Combining with (\ref{vol-2}) and (\ref{a-2-1}), and noting that $(2b)^{\frac{mp+q+a(q-m)}{p+q-a}}a^{-\frac{p(m-1)+a(q-m+1)}{p+q-m+1}}$ is uniformly bounded for $i$,
we derive
\begin{align*}
\int_{B_{2^{i+1}}} u^{p-a} |\nabla u|^q d\mu
&\lesssim i^{-\frac{mp+q+a(q-m)}{p+q-m+1}}
\sum_{k=i+1}^{2i}\int_{B_{2^{k+1}}\setminus B_{2^k}}2^{-k\frac{mp+q+a(q-m)}{p+q-m+1}}d\mu\notag\\
&\lesssim i^{-\frac{mp+q+a(q-m)}{p+q-m+1}}
\sum\limits_{k=i+1}^{2i} V(2^{k+1})
2^{-k\frac{mp+q+a(q-m)}{p+q-m+1}} \notag\\
&\lesssim i^{-\frac{mp+q+a(q-m)}{p+q-m+1}}
\sum\limits_{k=i+1}^{2i}
2^{k\left(m-\frac{mp+q+a(q-m)}{p+q-m+1}\right)}k^{m-1} \notag\\
&\lesssim i^{m-1-\frac{mp+q+a(q-m)}{p+q-m+1}}
\sum\limits_{k=i+1}^{2i}
2^{\frac{k(q-m)}{i(p+q-m+1)}},
\end{align*}
which is
\begin{align}\label{gu-2}
\int_{B_{2^{i+1}}} u^{p-m+1+\frac{1}{i}} |\nabla u|^q d\mu&\lesssim i^{m-1-\frac{mp+q+a(q-m)}{p+q-m+1}}\sum_{k=i+1}^{2i}2^{\frac{k(q-m)}{i(p+q-m+1)}}\nonumber\\
&\lesssim i^{\frac{q-m}{i(p+q-m+1)}},
\end{align}
Letting $i\rightarrow +\infty$ in (\ref{gu-2}), and applying Fatou's Lemma, we have
\begin{align*}
\int_{M} u^{p-m+1} |\nabla u|^q d\mu<\infty.
\end{align*}

Substituting $\varphi=\varphi_{i}$ into (\ref{2-5}),
and repeating the same procedure as in the proof of Theorem \ref{thm-non} \textbf{(I)}, we obtain
\begin{align}
\int_{M} u^{p-m+1} |\nabla u|^q d\mu=0,
\end{align}
which yields a contradiction with that $u$ is a nontrivial positive solution.

\textbf{In case (II-2)}, letting $u=e^v-1$ in (\ref{equ}), we obtain
\begin{align}\label{eq-change}
\Delta_m v + (m-1) |\nabla v|^{m} +(e^v-1)^p e^{v(q-m+1)} |\nabla v|^q \leq 0,
\end{align}
hence, we have
\begin{align*}
\Delta_m v+(m-1)|\nabla v|^m\leq0,
\end{align*}
By Theorem \ref{thm-non} \textbf{(I)},
we obtain that if (\ref{vol-2}) is satisfied, then $v\equiv const$, which implies that
$u\equiv const.$ However, this contradicts with that $u$ is a nontrivial positive solution.

\textbf{In case (II-3)}, we take the same procedure as in case (II-1) except by letting
\begin{align*}
a=m-1+\frac{1}{i}.
\end{align*}
Thus we complete the proof of Theorem \ref{thm-non} \textbf{(II)}.
\end{proof}

\vskip3ex
\begin{proof}[Proof of Theorem \ref{thm-non} \textbf{(III)}]
We separate the proof into three cases:
\begin{itemize}
\item[(III-1)] $(p, q)\in G_3\cap\{p>m-1-q\}$
\item[(III-2)] $(p, q)\in G_3\cap\{p=m-1-q\}$
\item[(III-3)] $(p, q)\in G_3\cap\{p<m-1-q\}$
\end{itemize}

\textbf{In case (III-1)}, since $(p,q)\in K_3$, let
\begin{align}\label{def-a-3}
a= \frac{p(1-m)}{q-m+1}+ \frac{1}{i}.
\end{align}
and $b$ be some large fixed constant.

Substituting $\psi=\varphi_i$ from (\ref{test-i}) into (\ref{est-1}), applying the same arguments as before,
we have
\begin{align*}
\int_M u^{p-a} |\nabla u|^q \varphi_i^b d\mu
\lesssim
\int_{\text{supp}|\nabla \varphi_i|} |\nabla \varphi_i|^{\frac{mp+q+a(q-m)}{p+q-m+1}} d\mu,
\end{align*}
Combining with (\ref{vol-3}) and (\ref{def-a-3}), which is
\begin{align*}
\int_{B_{2^{i+1}}} u^{p-a} |\nabla u|^q d\mu
&\lesssim i^{-\frac{mp+q+a(q-m)}{p+q-m+1}}
\sum_{k=i+1}^{2i}\int_{B_{2^{k+1}}\setminus B_{2^k}}2^{-k\frac{mp+q+a(q-m)}{p+q-m+1}}d\mu\notag\\
&\lesssim i^{-\frac{mp+q+a(q-m)}{p+q-m+1}}
\sum\limits_{k=i+1}^{2i} V(2^{k+1})
2^{-k\frac{mp+q+a(q-m)}{p+q-m+1})} \notag\\
&\lesssim i^{-\frac{mp+q+a(q-m)}{p+q-m+1}}
\sum\limits_{k=i+1}^{2i}
2^{k\left(\frac{q}{q-m+1}-\frac{mp+q+a(q-m)}{p+q-m+1}\right)}k^{\frac{m-1}{q-m+1}} \notag\\
&\lesssim i^{-\frac{mp+q+a(q-m)}{p+q-m+1}}
\sum\limits_{k=i+1}^{2i}
2^{-\frac{k(q-m)}{i(p+q-m+1)}}k^{\frac{m-1}{q-m+1}}\notag\\
&\lesssim i^{\frac{m-1}{q-m+1}-\frac{mp+q+a(q-m)}{p+q-m+1}+1},
\end{align*}
where we have used that
$$\frac{q}{q-m+1}-\frac{mp+q+a(q-m)}{p+q-m+1}=-\frac{q-m}{i(p+q-m+1)}.$$
Hence, we obtain
\begin{align*}\label{u3-1}
\int_{B_{2^{i+1}}} u^{\frac{pq}{q-m+1}+\frac{1}{i}} |\nabla u|^q d\mu
\lesssim
i^{-\frac{q-m}{i(p+q-m+1)}},
\end{align*}
Taking the limits as $i\to\infty$, and by Fatou's lemma, we obtain
\begin{align*}
\int_{M} u^{\frac{pq}{q-m+1}} |\nabla u|^q d\mu <\infty,
\end{align*}
Repeating the same argument as in proof of Theorem \ref{thm-non} \textbf{(I)}, we obtain
\begin{align*}
\int_M u^{\frac{pq}{q-m+1}} |\nabla u|^q d\mu = 0,
\end{align*}
which contradicts with that $u$ is a nontrivial positive solution.

\textbf{In case (III-2)}, taking $u=e^v-1$, we obtain that $v>0$ is also a nontrival positive solution to (\ref{eq-change}).
Noting $p+q=m-1$ in case \textbf{(III-2)}, and (\ref{change}) is equivalent to
\begin{eqnarray}\label{eq-trans2}
\Delta_m v+(m-1)|\nabla v|^m+\left(\frac{e^v}{e^v-1}\right)^{-p}|\nabla v|^q\leq0,
\end{eqnarray}
Since $p<0$, and $\left(\frac{e^v}{e^v-1}\right)^{-p}>1$, we have
\begin{align*}
\Delta_mv + |\nabla v|^q \leq 0,
\end{align*}
Noting $m-1<q<m$, and applying Theorem \ref{thm-non} \textbf{(I)},
we obtain that if (\ref{vol-3}) is satisfied, then $v\equiv const$, which implies that
$u\equiv const$. However, it contradicts with that $u$ is a nontrivial positive solution.

\textbf{In case (III-3)},
one can repeat the same argument as in case \textbf{(III-1)} except by substituting
\begin{align*}
a = \frac{p(1-m)}{q-m+1} -\frac{1}{i}.
\end{align*}
Hence, we complete the proof of Theorem \ref{thm-non} \textbf{(III)}.
\end{proof}
\vskip3ex
\begin{proof}[proof of theorem \ref{thm-non} \textbf{(IV)}]
Since here $q=m-1, p<0$,  take
\begin{align*}
a=l(m-1)+\frac{1}{i}, \quad b=-\frac{l(m-1)}{p}+m+\frac{1}{i},
\end{align*}
where $l>1$ is to be chosen later.

Substituting $\varphi_i$ into (\ref{est-1}), repeating the same procedures, we obtain
\begin{align*}
\int_M u^{p-a} |\nabla u|^q \varphi_i^b d\mu
\lesssim (2b)^{\frac{mp+q-a}{p}} a^{-(m-1)} \int_{M} |\nabla \varphi_i|^{\frac{mp+q-a}{p}} d\mu.
\end{align*}
It follows that
\begin{align*}
\int_{B_{2^{i+1}}} u^{p-a} |\nabla u|^q d\mu
&\lesssim (2b)^{\frac{mp+q-a}{p}}a^{-\frac{p(m-1)}{p}} i^{-\frac{mp+q-a}{p}}\nonumber\\
&\quad \times
\sum_{k=i+1}^{2i}\int_{B_{2^{k+1}}\setminus B_{2^k}}2^{-k\frac{mp+q-a}{p}}d\mu\notag\\
&\lesssim i^{-\frac{mp+q-a}{p}}
\sum\limits_{k=i+1}^{2i} V(2^{k+1})
2^{-k\frac{mp+q-a}{p}} \notag\\
&\lesssim i^{-\frac{mp+q-a}{p}}
\sum\limits_{k=i+1}^{2i}
2^{k\left(\alpha-\frac{mp+q-a}{p}\right)},
\end{align*}
Letting $l$ a fixed large enough constant such that for all $i$
$$\alpha-\frac{mp+q-a}{p}<0,$$
we obtain
\begin{eqnarray}
\int_{B_{2^{i+1}}} u^{p-a} |\nabla u|^q d\mu\lesssim i^{1-\frac{mp+q-a}{p}},
\end{eqnarray}
Further, we can also require that $l$ satisfying
$$1-\frac{mp+q-a}{p}<c<0.$$
then letting $i\to \infty$, we obtain
\begin{align*}
\int_Mu^{p-l(m-1)}|\nabla u|^qd\mu=0,
\end{align*}
which contradicts with that $u$ is nontrivial positive. Hence, we complete the proof for Theorem \ref{thm-non} \textbf{(IV)}.
\end{proof}

\vskip3ex

\begin{proof}[proof of Theorem \ref{thm-non} \textbf{(V)}]
When $(p, q)\in G_5$, taking the change $u=e^v$ (note that here $v$ may not be positive),
we obtain that $v$ satisfies
\begin{equation}\label{eq-trans3}
\Delta_m v + (m-1)|\nabla v|^m + |\nabla v|^q \leq 0,\quad\mbox{on $M$}.
\end{equation}
Noting in this case $q\leq m-1$.
For every $\lambda \in (m-1, m)$, applying Young's inequality, we obtain
\begin{align*}
|\nabla v|^\lambda \leq
\frac{\lambda-q}{m-q}|\nabla v|^m+\frac{m-\lambda}{m-q}|\nabla v|^q.
\end{align*}
It follows that
\begin{align*}
|\nabla v|^\lambda\leq
|\nabla v|^m + |\nabla v|^q.
\end{align*}
Hence from (\ref{eq-trans3}), we obtain
\begin{align}\label{eq-td}
\Delta_m v + l|\nabla v|^\lambda \leq 0, \quad\mbox{on $M$},
\end{align}
where $l=\min\{m-1,1\}$.

Define
\begin{align}\label{phiR}
\varphi_R(x) = h(\frac{|x|}{R}),
\end{align}
where $h$ is the same as in (\ref{def-h}).

Multiplying (\ref{eq-td}) by $\varphi_R^{z}$,
we have
\begin{align}
\int_{M}|\nabla v|^{\lambda} \varphi_R^zd\mu&\leq \frac{z}{l}\int_{M}|\nabla v|^{m-2}\varphi_R^{z-1}(\nabla v, \nabla\varphi_R)d\mu\nonumber\\
&\leq \frac{z}{l}\left(\int_M|\nabla v|^{\lambda}\varphi_R^z d\mu\right)^{\frac{m-1}{\lambda}}\left(\int_{M}|\nabla \varphi_R|^zd\mu\right)^{\frac{\lambda-m+1}{\lambda}},
\end{align}
where we take
$$z=\frac{\lambda}{\lambda-m+1}.$$

By the boundedness of $\int_{M}|\nabla v|^{\lambda} \varphi_R^zd\mu$, and $\varphi_R=1$ in $B_R$, and (\ref{vol-5}), we obtain
\begin{align}\label{G5-lim}
\int_{B_R}|\nabla v|^{\lambda}d\mu&\leq \left(\frac{z}{l}\right)^z\int_{M}|\nabla \varphi_R|^zd\mu\nonumber\\
&\lesssim \left(\frac{\rho z}{lR}\right)^zV(2R)\nonumber\\
&\lesssim \left(\frac{C_1z}{R}\right)^ze^{2\kappa R},
\end{align}
where we have also used that
$|\nabla \varphi_R|\leq\frac{\rho}{R}$, $C_1=\frac{\rho}{l}$.

Now let us connect $z$ and $R$ by defining
\begin{align}\label{zR}
z=\theta R,
\end{align}
where $\theta$ is a fixed positive constant to be determined later. It is easy to see that $\lambda\to (m-1)_{+}$ is equivalent to $R\to +\infty$.

Now let $\lambda\to (m-1)_{+}$ in (\ref{G5-lim}), by Fatou's lemma and (\ref{zR}),
\begin{align}\label{kappa-0}
\int_{M}|\nabla v|^{m-1}d\mu&\leq \lim_{\lambda\to (m-1)_{+}}
\int_{B_R}|\nabla v|^{\lambda}d\mu\lesssim \lim_{R\to \infty} \left(\frac{C_1z}{R}\right)^ze^{2\kappa R}\nonumber\\
&\asymp \lim_{R\to \infty}e^{R(2\kappa+\theta\ln(C_2\theta))}.
\end{align}
If we want
\begin{align}\label{kappa-theta}
2\kappa+\theta\ln(C_1\theta)<0,
\end{align}
which is equivalent to
\begin{align*}
e^{2\kappa}<\frac{1}{(C_1\theta)^{\theta}},
\end{align*}
Since $\frac{1}{(C_1\theta)^{\theta}}$ attains its maximum at $\theta=\frac{1}{C_1e}$.
Hence if $\kappa$ satisfies
\begin{align}\label{kappa-rough}
0<\kappa<\frac{1}{2C_1e}=\frac{\min\{m-1,1\}}{2\rho e},
\end{align}
there always exists $\theta>0$ such that (\ref{kappa-theta}) holds.


Under the above choice of $\kappa$ and $\theta$, from (\ref{kappa-0}), we obtain
\begin{align*}
\int_{M}|\nabla v|^{m-1}d\mu=0,
\end{align*}
which contradicts with that $u$ is a nontrivial positive solution.
Thus, we complete the proof for Theorem \ref{thm-non} \textbf{(V)}.
\end{proof}

\begin{remark}
\rm{By modifying function $h$, it is possible to let $\rho$ close to $1_{+}$ in (\ref{def-h})}.
\end{remark}

\vskip3ex
\begin{proof}[proof of theorem \ref{thm-non} \textbf{(VI)}] When $(p, q)\in G_6$, define
\begin{align*}
\Omega_k := \left\{ x \in M |0<u(x)\leq k, |\nabla u(x)| \neq 0 \right\}.
\end{align*}
Since $u$ is a nontrivial positive solution, without loss of generality, we can always assume that $\mu(\Omega_k)>0$ for some fixed large $k$.
Now let $v=\frac{u}{k}$, we know $v$ satisfies
$$\Delta_mv+k^{p+q-m+1}v^p|\nabla v|^q\leq0,\quad\mbox{on $M$}.$$
Let $\Omega= \left\{ x \in M |0< v(x)\leq 1, |\nabla v(x)| \neq 0 \right\}$, and hence $\mu(\Omega)>0$.

Recalling definition of $\Omega$, we obtain
\begin{align}\label{u4-Omega}
\int_{\Omega \cap B_R} \frac{|\nabla v|^q}{v} d\mu \leq \int_{B_R} v^{p-a} |\nabla v|^q d\mu.
\end{align}

Substituting $\varphi = \varphi_R$ in (\ref{phiR}) into (\ref{est-1}) but with $u$ replacing by $v$, and combining with (\ref{vol-6}), we obtain
\begin{align}\label{u4-1}
&\int_{B_R} v^{p-a} |\nabla v|^q d\mu\nonumber\\
&\lesssim (2b)^{\frac{mp+q+a(q-m)}{p+q-m+1}}a^{-\frac{p(m-1)+a(q-m+1)}{p+q-m+1}}k^{a-p-q}
\int_{\text{supp}|\nabla\varphi_R|}|\nabla\varphi_R|^{\frac{mp+q+a(q-m)}{p+q-m+1}}d\mu \notag\\
&\lesssim (2b)^{\frac{mp+q+a(q-m)}{p+q-m+1}}a^{-\frac{p(m-1)+a(q-m+1)}{p+q-m+1}}k^{a-p-q} V(2R) \left( \frac{1}{R} \right) ^{\frac{mp+q+a(q-m)}{p+q-m+1}} \notag\\
&\lesssim (2b)^{\frac{mp+q+a(q-m)}{p+q-m+1}}a^{-\frac{p(m-1)+a(q-m+1)}{p+q-m+1}}k^{a-p-q} e^{2\kappa R\ln (2R)} R^{-\frac{mp+q+a(q-m)}{p+q-m+1}},
\end{align}
Combining (\ref{u4-1}) with (\ref{u4-Omega}), we have
\begin{align}
\int_{\Omega \cap B_R} \frac{|\nabla v|^q}{v} d\mu \lesssim (2b)^{\frac{mp+q+a(q-m)}{p+q-m+1}}a^{-\frac{p(m-1)+a(q-m+1)}{p+q-m+1}}k^{a} e^{2\kappa R\ln (2R)}R^{-\frac{mp+q+a(q-m)}{p+q-m+1}},\nonumber\\
\end{align}
where the constant $k^{-p-q}$ is absorbed into $\lesssim$.

Choose
\begin{align*}
b=\frac{2(q-m)}{p+q-m+1}a,
\end{align*}
we know when $a$ is large enough, (\ref{cond-ab}) is satisfied for such choices of $a, b$.

Let us connect $a$ and $R$ by
\begin{align*}
a=2 R.
\end{align*}
Let us write
\begin{align*}
\frac{mp+q+a(q-m)}{p+q-m+1}&=c_1a+c_2,\nonumber\\
2b&=c_3a,\nonumber\\
-\frac{p(m-1)+a(q-m+1)}{p+q-m+1}&=c_4a+c_5,\nonumber\\
-\frac{mp+q+a(q-m)}{p+q-m+1}&=c_6a+c_7,\nonumber\\
\end{align*}
where
\begin{align*}
&c_1=\frac{q-m}{p+q-m+1},\quad c_2=\frac{mp+q}{p+q-m+1},\nonumber\\
&c_3=\frac{4(q-m)}{p+q-m+1},\nonumber\\
&c_4=-\frac{q-m+1}{p+q-m+1},\quad c_5=-\frac{p(m-1)}{p+q-m+1},\nonumber\\
&c_6=-\frac{q-m}{p+q-m+1},\quad c_7=-\frac{mp+q}{p+q-m+1}.
\end{align*}
Combining with the above, we have
\begin{align}\label{u4-1-1}
\int_{\Omega \cap B_{\frac{a}{2}}} \frac{|\nabla v|^q}{v} d\mu&\lesssim
e^{\frac{mp+q+a(q-m)}{p+q-m+1}\ln (2b)}e^{-\frac{p(m-1)+a(q-m+1)}{p+q-m+1}\ln a}\nonumber\\
&\quad \times e^{a\ln k+2\kappa R\ln (2R)}e^{-\frac{mp+q+a(q-m)}{p+q-m+1}\ln R}\nonumber\\
&\lesssim e^{(c_1a+c_2)\ln(c_3a)}e^{(c_4a+c_5)\ln a}\nonumber\\
&\quad \times e^{a\ln k+\kappa a\ln a+(c_6a+c_7)\ln\frac{a}{2}}\nonumber\\
&\lesssim e^{C_1a\ln a+C_2a+C_3\ln a+C_4}
\end{align}
where
\begin{align}\label{kappa-value}
&C_1=c_1+c_4+\kappa+c_6=\kappa+\frac{m-1-q}{p+q-m+1},\\
&C_2=c_1\ln c_3+\ln k-c_6\ln 2,\nonumber\\
&C_3=c_2+c_5+c_7,\nonumber\\
&C_4=c_2\ln c_3-c_7\ln 2.\nonumber
\end{align}
From (\ref{kappa-value}), we know if
$$0<\kappa<\frac{m-1-q}{m-1-p-q},$$
then $C_1<0$, by letting $a\rightarrow+\infty$ in (\ref{u4-1-1}), we obtain
\begin{align*}
\int_{\Omega} \frac{|\nabla v|^q}{v} d\mu=0,
\end{align*}
which contradicts with the definition of $\Omega$. Hence, we complete the proof for Theorem \ref{thm-non} \textbf{(VI)}.
\end{proof}

\section{Counter Examples}
We take the change
\begin{align}\label{change}
v = \left\{
\begin{array}{ll}
\int_0^{u}\exp{\left(\frac{s^{p+1}}{(p+1)(m-1)}\right)}ds, &\quad\text{$p\neq-1$}, \\
u^{\frac{m}{m-1}}, &\quad\text{$p=-1$},
\end{array}
\right.
\end{align}
The above change is well defined, since both the ranges of $\int_0^{t}\exp{\left(\frac{s^{p+1}}{(p+1)(m-1)}\right)}ds$
and $t^{\frac{m}{m-1}}$ are $[0,\infty)$, see \cite{HB15} and also the references therein.

\vskip1ex
Based on the above change, it is easy to verify the following lemma.
\begin{lemma}\label{lem-m}
\rm{If $v$ is a positive solution to $\Delta_m v\leq0$, then $u$ in (\ref{change}) is a positive solution to
$\Delta_m u+u^p|\nabla u|^m\leq0$.}
\end{lemma}

\begin{proof}[Proof of Theorem \ref{thm-ex} \textbf{(B)} of case $q=m$]

If (\ref{vol-2-1}) are satisfied, then by \cite[Theorem 2.5]{Holo00},
we know there exists a noncompact complete manifold on which $\Delta_m v\leq0$ admits a nontrivial positive solution. Combining with Lemma \ref{lem-m}, we know
there exists a nontrivial positive solution $u$ to $\Delta_m u+u^p|\nabla u|^m\leq0$. Hence, we complete the proof.
\end{proof}

Thus, we ignore the case of $q=m$ in the following.

\begin{proof}[Proof of Theorem \ref{thm-ex} \textbf{(A)} and \textbf{(B)}]
Let $(\mathbb{R}^n, g)$ be a Riemannian manifold with Riemannian metric
$g=dr^2+\psi(r)^2d\theta^2$,
where $(r, \theta)$ is the polar coordinates in $\mathbb{R}^n$,
and $\psi(r)$ is a smooth, positive increasing function on $(0, +\infty)$ such that
\begin{align*}
\psi(r) = \left\{
\begin{array}{ll}
r, &\quad\text{for small $r$}, \\
c_0\left(r^{\alpha-1} (\ln r)^\beta\right)^{\frac{1}{n-1}}, &\quad\text{for large $r$},
\end{array}
\right.
\end{align*}
where
\begin{align}\label{def-alpha}
\alpha=\left\{
\begin{array}{ll}
\frac{mp+q}{p+q-m+1}, &\quad (p, q)\in G_1, \\
m, &\quad (p, q)\in G_2\setminus\{q=m\},
\end{array}
\right.
\end{align}
and
\begin{align}\label{def-beta}
\beta>\left\{
\begin{array}{ll}
\frac{m-1}{p+q-m+1}, &\quad (p, q)\in G_1, \\
m-1, &\quad (p, q)\in G_2\setminus\{q=m\},
\end{array}
\right.
\end{align}
and the constant $c_0$ is to make $\psi(r)$ be an increasing function.

Choose $o$ to be the origin point, and the surface area $S$ of the ball $B_r$ can be determined by
\begin{align*}
S(r)=
\omega_n
\left\{
\begin{array}{ll}
r^{n-1}, &\quad\text{for small $r$}, \\
c_0^{n-1}r^{\alpha-1} (\ln r)^\beta, &\quad\text{for large $r$}.
\end{array}
\right.
\end{align*}
where $\omega_n$ is the surface area of unit ball in $\mathbb{R}^n$.

Since
\begin{align*}
V(r) =\int_0^r S(\tau)d\tau.
\end{align*}
It is easy to verify that, for all large enough $r$
\begin{align*}
V(r) \lesssim r^\alpha(\ln r)^\beta.
\end{align*}

Since the solution $u$ to be constructed here is radial, let us denote by $u(r)$,
then (\ref{equ}) leads to
\begin{align}\label{rad}
(S|u^{\prime}|^{m-2}u^{\prime})^{\prime}+Su^{p}|u^{\prime}|^q\leq 0.
\end{align}
In the following, we will construct the solution in three steps: the first step is to construct the solution near infinity, the second step is to construct the solution near origin, and the last step is to glue the two parts solution in a proper way.  Our method is motivated by the gluing technique used in \cite{Grigoryan13} and \cite{Wang-Xiao}.

\textbf{Step 1, solution near infinity:\;}
Set
\begin{align}\label{rad-inf}
u_{\eta}(r)= \left\{
\begin{array}{ll}
\int_r^{\infty}S(s)^{-\frac{1}{m-1}}(\ln s)^{\frac{\eta}{m-1}}ds, & (p, q)\in G_1, \\
(\ln r)^{-\frac{\eta}{m-1}}, & (p, q)\in G_2\setminus\{q=m\},
\end{array}	
\right.
\end{align}
for $r>R_0$, where $\eta, R_0$ are constants to be chosen later. According to the definition of $S(r)$, we know that the above
$u_\eta$ is well defined.

We will show $u_\eta$ is a solution of (\ref{rad}) on $[R_0,\infty)$ for some $\eta$.
Since
\begin{align*}
u^{\prime}_{\eta} = \left\{
\begin{array}{ll}
-S(r)^{-\frac{1}{m-1}}(\ln r)^{\frac{\eta}{m-1}}, & (p, q)\in G_1, \\
-\frac{\eta}{m-1} r^{-1} (\ln r)^{\frac{1-\eta-m}{m-1}}, & (p, q)\in G_2\setminus\{q=m\},
\end{array}
\right.
\end{align*}
and
\begin{align}\label{inf-est1}
&\left(S|u^{\prime}_{ \eta}|^{m-2}u^{\prime}_{ \eta}\right)^{\prime} \notag\\
&= \left\{
\begin{array}{ll}
-\eta r^{-1}(\ln r)^{\eta-1}, & (p, q)\in G_1, \\
\omega_nc_0^{n-1}(\eta+m-1-\beta) \left(\frac{\eta}{m-1}\right)^{m-1} r^{-1} (\ln r)^{\beta -m-\eta}, & (p, q)\in G_2\setminus\{q=m\},
\end{array}
\right.
\end{align}
and
\begin{align}\label{inf-est2}
Su_\eta^{p}|u_\eta^{\prime}|^q \lesssim \left\{
\begin{array}{ll}
r^{-1} (\ln r)^{\frac{\eta (p+q)-\beta (p+q-m+1)}{m-1}}, & (p, q)\in G_1, \\
\eta ^q r^{m -q-1} (\ln r)^{\frac{\beta (m-1)-q (\eta +m-1)-\eta p}{m-1}}, & (p, q)\in G_2\setminus\{q=m\},
\end{array}
\right.
\end{align}

$\bullet$ \textbf{Case of $(p,q)\in G_1$}.
Recalling $\beta > \frac{m-1}{p+q-m+1} $, we tak $\eta$ such that
$$\frac{\eta (p+q)-\beta (p+q-m+1)}{m-1}<\eta-1, $$
more precisely,
$$0<\eta<\beta-\frac{m-1}{p+q-m+1}. $$
Combining with (\ref{inf-est1}) and (\ref{inf-est2}), we have for large $r$,
\begin{align}\label{inf-ver-g1+}
(S|u_\eta^{\prime}|^{m-2}u_\eta^{\prime})^{\prime}+Su_\eta^{p}|u_\eta^{\prime}|^q\leq 0. \quad\mbox{ for $(p, q)\in G_1$}.
\end{align}

$\bullet$ \textbf{Case of $(p, q)\in G_2\setminus \{q=m\}$}. Since $\beta > m-1$, let us choose $\eta$ such that
\begin{align*}
\eta+m-1-\beta < 0
\end{align*}
which is
\begin{align*}
0<\eta<\beta-m-1
\end{align*}

Since $(p, q)\in G_2\setminus \{q=m\}$, it follows that $q>m$.
Noting $m-q-1<-1$, combining with (\ref{inf-est1}) and (\ref{inf-est2}), there exists $R_0$ such that we for all $r\geq R_0$,
\begin{align}\label{inf-ver-g2}
(S|u_\eta^{\prime}|^{m-2}u_\eta^{\prime})^{\prime}+Su_\eta^{p}|u_\eta^{\prime}|^q\leq 0. \quad \mbox{for $(p, q)\in G_2\setminus \{q=m\}$}.
\end{align}

Combining (\ref{inf-ver-g1+}) and (\ref{inf-ver-g2}),
we derive $u_\eta$ as in (\ref{rad-inf}) is a solution of (\ref{rad})
for $r\in [R_0, +\infty)$. Moreover, in both cases, we know
\begin{align}\label{u-der<0}
u_\eta^{\prime}(r) < 0, \quad\mbox{$r\in [R_0, +\infty)$}.
\end{align}

\textbf{Step 2, solution near origin:\;}
Consider the following ordinary differential equation
\begin{align}\label{ori-ode}
(S|u^{\prime}|^{m-2}u^{\prime})^{\prime}+\lambda Su^{p}|u^{\prime}|^q\leq0,
\end{align}
where $\lambda$ is to be chosen later.

We construct different $u$ of (\ref{ori-ode}) for $(p,q)\in G_1$ and $(p,q)\in G_2\setminus\{q=m\}$ respectively.

\textbf{$\bullet$ Case of $(p, q)\in G_1$},
let
\begin{align}\label{rad-ori}
u_{1,\rho}(r)=A_{1,\rho}\int_{r}^{\rho}\left(1-e^{-\frac{x}{\rho}}\right)^{\theta}dx,
\end{align}
where $r\in[0, \rho)$, $\rho>R_0$, and $\theta>0$ is to be chosen later, and
$A_{1,\rho}$ is a constant such that $u_{1,\rho}(0)=1$.
Then we have
\begin{align}\label{u1-1}
(S|u^{\prime}_{1,\rho}|^{m-2}u^{\prime}_{1,\rho})^{\prime}&=-A^{m-1}_{1,\rho}\left(S(1-e^{-\frac{r}{\rho}})^{\theta(m-1)}\right)^{\prime}\nonumber\\
&=-\frac{A^{m-1}_{1,\rho}\theta(m-1)}{\rho}S(1-e^{-\frac{r}{\rho}})^{\theta(m-1)-1}e^{-\frac{r}{\rho}}\nonumber\\
&\quad-A^{m-1}_{1,\rho}S^{\prime}(1-e^{-\frac{r}{\rho}})^{\theta(m-1)}\nonumber\\
&\leq-\frac{A^{m-1}_{1,\rho}\theta(m-1)}{\rho}S(1-e^{-\frac{r}{\rho}})^{\theta(m-1)-1}e^{-\frac{r}{\rho}} \notag\\
&\leq -\frac{A^{m-1}_{1,\rho}\theta(m-1)}{\rho}S(1-e^{-\frac{r}{\rho}})^{\theta(m-1)-1}.
\end{align}
where we have used that $S$ is an increasing function.

When $(p, q)\in G_1$, we have $p\geq0$.
Noting that $0\leq u_{1,\rho}\leq1$, we have
\begin{align}\label{u1-2}
\lambda Su_{1,\rho}^{p}|u_{1, \rho}^{\prime}|^q \leq
\lambda A_{1,\rho}^q S (1-e^{-\frac{r}{\rho}})^{q\theta}.
\end{align}
To suffice that $u_{1,\rho}$ is a solution to (\ref{ori-ode}) on $(0,\rho)$, from (\ref{u1-1}) and (\ref{u1-2}), we require that
\begin{align*}
-\frac{A^{m-1}_{1,\rho}\theta(m-1)}{\rho}S(1-e^{-\frac{r}{\rho}})^{\theta(m-1)-1}+\lambda A_{1,\rho}^q S (1-e^{-\frac{r}{\rho}})^{q\theta}\leq0,
\end{align*}
which yields
\begin{align*}
0<\lambda\leq\frac{\theta(m-1) A_{1, \rho}^{m-1-q}(1-e^{-\frac{r}{\rho}})^{\theta(m-1-q)-1}}{\rho},\quad\mbox{for all $r\in(0,\rho)$}.
\end{align*}
To guarrtee the existence of such $\lambda$ in the above, we need that
\begin{align}\label{theta1}
\theta \left\{
\begin{array}{ll}
=\mbox{any positive value}, \quad \mbox{when $q\geq m-1$}, \\
<\frac{1}{m-1-q}, \quad\mbox{when $q<m-1$},
\end{array}
\right.
\end{align}
hence, we can choose
\begin{align}\label{lambda-1-rho}
0<\lambda\leq\frac{\theta(m-1)(1-e^{-1})^{\theta(m-1-q)+1}}{\rho A_{1, \rho}^{q-m+1}}.
\end{align}
To suffice $u_{1,\rho}^p|\nabla u|^q\in L^1(B_{\epsilon})$, where $B_{\epsilon}$ is a small ball centered at $o$.
We also need that
\begin{align*}
\theta q+n>0,
\end{align*}
Combining with (\ref{theta1}), we can choose $\theta$ satisfying
\begin{align*}
\theta \left\{
\begin{array}{ll}
=\mbox{any positive value}, \quad q\geq m-1 \\
<\frac{1}{m-1-q}, \quad 0\leq q<m-1,\\
<\min\{\frac{1}{m-1-q}, -\frac{n}{q}\}, \quad q<0.
\end{array}
\right.
\end{align*}
Under the above choice of $\lambda, \theta$, we know $u_{1,\rho}$ is a $C^1$ solution to (\ref{ori-ode}) on $(0,\rho)$. Besides,
\begin{align}\label{A1-rho-1}
\lim\limits_{\rho\to+\infty}A_{1, \rho}=\lim\limits_{\rho\to+\infty}\left(\int_{0}^{\rho}(1-e^{-\frac{x}{\rho}})^{\theta}dx\right)^{-1}=0,
\end{align}
It follows that
\begin{equation}\label{u1-rho-1}
\lim\limits_{\rho\to+\infty}u_{1,\rho}^{\prime}(r)=0,
\end{equation}
Moreover, applying L'Hospital's rule, we have
\begin{equation}\label{u1-rho-2}
\lim\limits_{\rho\to+\infty}u_{1,\rho}(r)=\lim\limits_{\rho\to+\infty}
\frac{u_{1, \rho}(r)}{u_{1, \rho}(0)}=\lim\limits_{\rho\to+\infty}
\frac{\int_{r}^{\rho}(1-e^{-\frac{x}{\rho}})^{\theta}dx}{\int_{0}^{\rho}(1-e^{-\frac{x}{\rho}})^{\theta}dx}=1,
\end{equation}

\textbf{$\bullet$ Case of $(p,q)\in G_2\setminus\{q=m\}$}, let
\begin{align}\label{rad-ori-2}
u_{2,\rho}(r)=A_{2,\rho}\int_{r}^{2\rho}\left(1-e^{-\frac{x}{2\rho}}\right)^{\theta}dx,
\end{align}
where $r\in[0, \rho]$, $\rho\geq R_0$, $\theta>0$, and
$A_{2,\rho}$ is a constant such that $u_{2, \rho}(0)=1$.
Similarly
\begin{align*}
(S|u^{\prime}_{2,\rho}|^{m-2}u^{\prime}_{2,\rho})^{\prime}\leq -\frac{A^{m-1}_{2,\rho}\theta(m-1)}{2\rho}S(1-e^{-\frac{r}{2\rho}})^{\theta(m-1)-1}e^{-\frac{1}{2}},
\end{align*}
Let us denote
$$L= \int_{\rho}^{2\rho}\left(1-e^{-\frac{x}{2\rho}}\right)^{\theta}dx.$$
Note
\begin{align*}
0<LA_{2,\rho}\leq u_{2, \rho}\leq 1, \quad |u^{\prime}_{2, \rho}|= A_{2,\rho}(1-e^{-\frac{r}{2\rho}})^{\theta},
\end{align*}
we obtain
\begin{align*}
\lambda Su_{2,\rho}^{p}|u_{2, \rho}^{\prime}|^q \leq
\lambda S \max\{1, L^pA_{2, \rho}^p\}A_{2, \rho}^q(1-e^{-\frac{r}{2\rho}})^{q\theta}.
\end{align*}
To suffice that $u_{2, \rho}$ is a solution to (\ref{ori-ode}) on $[0,\rho]$, by applying the same argument as in case of
$(p, q)\in G_1$, we need that
\begin{align}\label{lambda}
0<\lambda\leq
\frac{A_{2, \rho}^{m-1-q}\theta(m-1)}{2\rho\max\{1, L^p A_{2, \rho}^p\}}(1-e^{-\frac{r}{2\rho}})^{\theta(m-1-q)-1}e^{-\frac{1}{2}},
\end{align}
Since $q>m$ in the case of $(p,q)\in G_2\setminus\{q=m\}$, it follows that the existence of $\lambda>0$ in (\ref{lambda}),
and $\lambda$ can be chosen satisfying
\begin{equation*}
0<\lambda\leq
\frac{A_{2, \rho}^{m-1-q}\theta(m-1)}{2\rho\max\{1, L^pA_{2, \rho}^p\}}(1-e^{-\frac{1}{2}})^{\theta(m-1)-1-q\theta}e^{-\frac{1}{2}}.
\end{equation*}
Under the above choice of $\lambda$, then we know $u_\rho$ is a $C^1$ solution of (\ref{ori-ode}) for $r\in[0, \rho]$.
Besides,
\begin{align*}
\lim\limits_{\rho\to+\infty}A_{2, \rho}=\lim\limits_{\rho\to+\infty}\left(\int_{0}^{2\rho}(1-e^{-\frac{x}{2\rho}})^{\theta}dx\right)^{-1}=0,
\end{align*}
and
\begin{equation}\label{u-ori-der}
\lim\limits_{\rho\to+\infty}u_{2,\rho}^{\prime}(r)=0,
\end{equation}
and
\begin{equation}\label{u-ori-1}
\lim\limits_{\rho\to+\infty}u_{2,\rho}(r)=\lim\limits_{\rho\to+\infty}
\frac{u_{2, \rho}(r)}{u_{2, \rho}(0)}=\lim\limits_{\rho\to+\infty}
\frac{\int_{r}^{2\rho}(1-e^{-\frac{x}{2\rho}})^{\theta}dx}{\int_{0}^{2\rho}(1-e^{-\frac{x}{2\rho}})^{\theta}dx}=1,
\end{equation}

\vskip2ex
Now we come to the final step to show how to glue $u_{\eta}$ and $u_{i,\rho} (i=1,2)$ in a proper way.

\textbf{Step 3, gluing the solution:\;}
From (\ref{u-der<0}), we have
\begin{align}
\frac{u^{\prime}_{\eta}(R_0)}{u_{\eta}(R_0)}<0.
\end{align}
From (\ref{u1-rho-1})-(\ref{u1-rho-2}), and (\ref{u-ori-der})-(\ref{u-ori-1}), we obtain
\begin{align}\label{der-1}
\lim_{\rho\to+\infty}\frac{u_{i, \rho}^{\prime}(R_0)}{u_{i, \rho}(R_0)}=0,\quad\mbox{for $i=1,2$}.
\end{align}

We claim there exists $\rho_0>R_0$ such that
\begin{align}\label{claim}
\frac{u_{i,\rho_0}^{\prime}(R_0)}{u_{i, \rho_0}(R_0)}=\frac{u^{\prime}_{\eta}(R_0)}{u_{\eta}(R_0)},\quad\mbox{for $i=1,2$}.
\end{align}

$\bullet$ \textbf{In case of $(p,q)\in G_1$}, we have
\begin{align}\label{der-2}
\lim_{\rho\downarrow R_0}\frac{u_{1, \rho}^{\prime}(R_0)}{u_{1,\rho}(R_0)}=-\infty.
\end{align}
Let us fix $R_0$, we know $\frac{u_{\rho}^{\prime}(R_0)}{u_{\rho}(R_0)}$ is a continuous function of $\rho_0$ in $(R_0, +\infty)$.
Combining with (\ref{der-1}) and (\ref{der-2}), we know there exists a $\rho_0>R_0$ such that
\begin{align*}
\frac{u_{1,\rho_0}^{\prime}(R_0)}{u_{1, \rho_0}(R_0)}=\frac{u^{\prime}_{\eta}(R_0)}{u_{\eta}(R_0)}.
\end{align*}
Fix this $\rho_0$, then there exists some $\tau\in (0, +\infty)$ such that
\begin{align*}
u_{\eta}(R_0)=\tau u_{1,\rho_0}(R_0), \quad u_{\eta}^{\prime}(R_0)=\tau u_{1,\rho_0}^{\prime}(R_0).
\end{align*}

$\bullet$ \textbf{In case of $(p, q)\in G_2\setminus\{q=m\}$}, we have
\begin{align*}
\frac{u_{2, \rho}^{\prime}(R_0)}{u_{2, \rho}(R_0)}=-\frac{(1-e^{-\frac{R_0}{2\rho}})^{\theta}}{\int_{R_0}^{2\rho}(1-e^{-\frac{x}{2\rho}})^{\theta}dx}
\end{align*}
and
\begin{align*}
\frac{u_{\eta}^{\prime}(R_0)}{u_{\eta}(R_0)}=-\frac{\eta}{m-1}R_0^{-1}(\ln R_0)^{-1},
\end{align*}
We will show there exists $\rho$ such that
\begin{align}\label{urho-2-0}
\frac{u_{\eta}^{\prime}(R_0)}{u_{\eta}(R_0)}>\frac{u_{2, \rho}^{\prime}(R_0)}{u_{2, \rho}(R_0)}.
\end{align}
The above is equivalent to
\begin{align}\label{urho-2-1}
\frac{\eta}{m-1}R_0^{-1}(\ln R_0)^{-1} \int_{R_0}^{2\rho}(1-e^{-\frac{x}{2\rho}})^{\theta}dx<(1-e^{-\frac{R_0}{2\rho}})^{\theta},
\end{align}
Let $\rho=R_0$, we know
\begin{align}\label{urho-2-2}
\mbox{LHS of (\ref{urho-2-1})}\leq \frac{\eta}{m-1}(\ln R_0)^{-1}(1-e^{-1})^{\theta},
\end{align}
and
\begin{align}\label{urho-2-3}
\mbox{RHS of (\ref{urho-2-1})}=(1-e^{-\frac{1}{2}})^{\theta},
\end{align}
Combining with (\ref{urho-2-2}) and (\ref{urho-2-3}), we know, if we choose $R_0$ large enough, then (\ref{urho-2-0}) is satisfied.

Hence, for such fixed $R_0$, we have
\begin{align}
0=\lim_{\rho\to+\infty}\frac{u_{2, \rho}^{\prime}(R_0)}{u_{2, \rho}(R_0)}>\frac{u_{\eta}^{\prime}(R_0)}{u_{\eta}(R_0)}>\frac{u_{2, R_0}^{\prime}(R_0)}{u_{2, R_0}(R_0)}
\end{align}
Since $\frac{u_{2, \rho}^{\prime}(R_0)}{u_{2, \rho}(R_0)}$ is a continuous function of $\rho$, hence there exists some
$\rho>R_0$ such that
\begin{align*}
\frac{u_{2, \rho}^{\prime}(R_0)}{u_{2, \rho}(R_0)}=\frac{u_{\eta}^{\prime}(R_0)}{u_{\eta}(R_0)}.
\end{align*}
hence, we finish the claim (\ref{claim}).

In both cases, from (\ref{claim}), there exists some $\tau=\tau(i)>0$ such that for
\begin{align*}
u_{\eta}(R_0)=\tau u_{i,\rho}(R_0),\quad u_{\eta}^{\prime}(R_0)=\tau u_{i,\rho}^{\prime}(R_0).
\end{align*}
Define $u$ as follows
\begin{align}\label{u-const}
u=\left\{
\begin{array}{ll}
\tau u_{i,\rho_0}, \quad r\in [0, R_0], \\
u_{\eta}, \quad r\in (R_0, +\infty).
\end{array}
\right.
\end{align}
Hence $u(r)\in C^1(0,\infty)$.
Moreover, $u\in W_{loc}^{1,m}(M)$ satisfies
\begin{equation}\label{glue-1}
\Delta_m u+\lambda_{\rho_0} \tau^{m-1-p-q}u^{p}|\nabla u|^q\leq0, \quad\mbox{in $B_{R_0}$},
\end{equation}
and also
\begin{equation}\label{glue-2}
\Delta_m u+u^{p}|\nabla u|^q\leq0, \quad\mbox{in $M\setminus B_{R_0}$}.
\end{equation}
where $\lambda_{\rho_0}$ satisfies (\ref{lambda-1-rho}) for $\rho=\rho_0$.

Combining with (\ref{glue-1}) and (\ref{glue-2}), we derive that $u$ is a weak positive solution to
\begin{equation}
\Delta_m u+\delta u^{p}|\nabla u|^q\leq0,\quad\mbox{on $M$},
\end{equation}
where $\delta=\min\{\lambda_{\rho} \tau^{m-1-p-q}, \;1\}$.

For the case of $p+q\neq m-1$, making the change
$u\to c_1u$ where
$$c_1=\delta^{-\frac{1}{ p+q-m+1}}.$$
Hence, we obtain a nontrivial positive weak solution $u$ to (\ref{equ}) on $M$.

While for the case of $p+q=m-1$, let us make the change $u=v^{a}$, where $a$ is to be chosen later.
It is easy to verify that $v$ satisfy
\begin{align*}
\Delta_m v+\delta a^{-p}v^p|\nabla v|^q\leq(1-a)(m-1)v^{-1}|\nabla v|^m.
\end{align*}
Since under this case $-p>1$, we choose $a$ large enough such that $a>1$ and $\delta a^{-p}\geq1$, we obtain that $v$ satisfies
\begin{equation*}
\Delta_m v+v^p|\nabla v|^q\leq0.
\end{equation*}
which yields $v$ is a nontrivial positive weak solution to (\ref{equ}) on $M$. Thus, we complete proof of Theorem \ref{thm-ex}
\textbf{(A)} and \textbf{(B)}.
\end{proof}

\vskip3ex

\begin{proof}[proof of Theorem \ref{thm-ex} \textbf{(C)}]
By Theorem \ref{thm-ex} \textbf{(A)},
we know when $q>m-1$, there exists a manifold $M$ satisfying (\ref{vol-3}), i.e.
\begin{align*}
V(r) \lesssim r^{\frac{q}{q-m+1}} (\ln r)^{\frac{1}{q-m+1} +\epsilon},\quad\mbox{for all large $r$},
\end{align*}
which admits a nontrivial positive solution $v$ to
\begin{align*}
\Delta_mv + |\nabla v|^q \leq 0.
\end{align*}
Define
\begin{align*}
u=v+1.
\end{align*}
Since $(p,q)\in G_3$, then it is easy to check that $u$ is a nontrivial positive weak solution to (\ref{equ}) on $M$.
\end{proof}

\vskip3ex
\begin{proof}[Proof of Theorem \ref{thm-ex} \textbf{(D)}]
We still work on the model manifold $(\mathbb{R}^n, g)$ as in the proof of Theorem \ref{thm-ex} \textbf{(A)} and \textbf{(B)}, but with different metric $g$ such that
\begin{align*}
S(r)=e^{\lambda r},\quad \mbox{as for all $r\geq R_0$}.
\end{align*}
and
\begin{align*}
V(r)\lesssim e^{\lambda r}, \quad \mbox{as for all $r\geq R_0$}.
\end{align*}
where $R_0$ will be determined later.

We construct two parts of the solution: one is near infinity, and the other is near origin point, then we glue the two parts in a proper way.

\textbf{Solution near infinity:}
Let $u_{\eta}(r)=\lambda^{\frac{1}{p}}+r^{-\eta}$, where $\eta>0$ is to be chosen later. Noting
\begin{align*}
(S|u^{\prime}|^{m-2}u^{\prime})^{\prime}&=-e^{\lambda r}\eta^{m-1}r^{-(\eta+1)(m-1)}\nonumber\\
&\quad +e^{\lambda r}\eta^{m-1}(\theta+1)(m+1)r^{-(\eta+1)(m-1)-1}
\end{align*}
and
\begin{align*}
Su^p|\nabla u|^{m-1}=e^{\lambda r}(\lambda^{\frac{1}{p}}+r^{-\eta})^p\eta^{m-1}r^{-(\eta+1)(m-1)},
\end{align*}

Clearly, $(S|u^{\prime}|^{m-2}u^{\prime})^{\prime}+Su^p|\nabla u|^{m-1}\leq0$ is equivalent to
\begin{align}
e^{\lambda r}\eta^{m-1}r^{-(\eta+1)(m-1)-1}[-r+(\eta+1)(m-1)+(\lambda^{\frac{1}{p}}+r^{-\eta})^pr]\leq0
\end{align}

Noting that $(\lambda^{\frac{1}{p}}+r^{-\eta})^p=\lambda^{\frac{1}{p}}+pr^{-\eta}+o(r^{-2\eta})$, we obtain that
if $0<\eta<1$, then
\begin{align}
\lim_{r\to\infty}[-r+(\eta+1)(m-1)+(\lambda^{\frac{1}{p}}+r^{-\eta})^pr]=-\infty.
\end{align}
Thus, there exists some large $R_0$ such that $u_{\eta}$ is a solution to (\ref{ori-ode}) near infinity.

\textbf{Solution near origin:} When $q=m-1$, applying the same arguments as in the proof of Theorem \ref{thm-ex} \textbf{(B)}, one can verify that $u_{2,\rho}(r)$ as defined in (\ref{rad-ori-2}) is the solution to (\ref{ori-ode}) with a modification of $\rho$.

Repeating the same procedure as in proof of Theorem \ref{thm-ex} for the case of $G_2\setminus \{q=m\}$, one can construct
that $u$ as in (\ref{u-const}) is a nontrivial weak positive solution to (\ref{equ}).
Hence, we finish the proof of \ref{thm-ex} \textbf{(D)}.
\end{proof}
\vskip3ex
\begin{proof}[Proof of Theorem \ref{thm-ex} \textbf{(E)}]
We continue to work on the model manifold $(\mathbb{R}^n, g)$ as in the proof of Theorem \ref{thm-ex} \textbf{(A)} and \textbf{(B)}, but with different metric $g$ such that
\begin{align}
V(r)= e^{\iota r},\quad \mbox{for all $r\geq R_0$},
\end{align}
and
\begin{align}
S(r)=\iota e^{\iota r},\quad \mbox{as for all $r\geq R_0$}.
\end{align}
where $\iota, R_0$ are to be chosen later.

We still construct two parts of the solution, but we need to deal with the case of $p>0, q=m-1-p$ and $p=0, q=m-1$ individually.

\vskip1ex

\textbf{(1). Case of $p>0, q=m-1-p$}.

\textbf{Solution near infinity:}
Let $R_0$ satisfy
\begin{align}\label{5-00}
R_0<\frac{(p+1)^{\frac{p}{p+1}}(m-1)^{\frac{1}{p+1}}}{p}.
\end{align}
We take
\begin{align}\label{5-0}
u_{\eta}(r) = e^{-\eta r},\quad\mbox{for all $r\geq R_0$}
\end{align}
where $\eta$ is to be determined later.

Since
\begin{align}\label{5-1}
(S|u_{\eta}^{\prime}|^{m-2}u_{\eta}^{\prime})^{\prime}&=-\left(\iota\eta^{m-1}e^{[\iota -\eta (m-1)]r}\right)^{\prime}\nonumber\\
&=-\iota\eta^{m-1}[\iota-\eta (m-1)]e^{[\iota-\eta (m-1)]r},
\end{align}
and
\begin{align}\label{5-2}
Su_{\eta}^p|u_{\eta}^{\prime}|^q=\iota\eta^q e^{[\iota-\eta(m-1)]r},
\end{align}
If $u_{\eta}$ is a solution to (\ref{rad}) near infinity, combining with (\ref{5-1}) and (\ref{5-2}), we arrive
\begin{align*}
\eta^q-\eta^{m-1}[\iota -\eta (m-1)]\leq0,
\end{align*}
which is equivalent to
\begin{align}\label{5-3}
\iota\geq \eta^{-p}+\eta(m-1),
\end{align}
Since the RHS of (\ref{5-3}) attains minimum at $\eta=\left(\frac{p}{m-1}\right)^{\frac{1}{p+1}}$, hence we know if
\begin{align}\label{5-3-iota}
\iota\geq H(p):=\left(\frac{m-1}{p}\right)^\frac{p}{p+1} + (m-1) \left(\frac{m-1}{p}\right)^{\frac{-1}{p+1}},
\end{align}
there always exists a $\eta>0$ such that (\ref{5-0}) is a solution to (\ref{rad}) near infinity.

\textbf{Solution near origin:} Let
\begin{align}
u_{0}(r)=c_0-\frac{c_0(pr)^{\frac{p+1}{p}}}{(p+1)(m-1)^{\frac{1}{p}}},\quad\mbox{$r\in[0, \frac{(p+1)^{\frac{p}{p+1}}(m-1)^{\frac{1}{p+1}}}{p})$},
\end{align}
where $c_0$ is to be chosen later.
It is easy to verify that
\begin{align}\label{5-5}
(S|u_{0}^{\prime}|^{m-2}u_0^{\prime})^{\prime}&=(m-1)S(-u_0^{\prime})^{m-2}u_0^{\prime\prime}-S^{\prime}(-u_0^{\prime})^{m-1}\nonumber\\
&\leq (m-1)S(-u_0^{\prime})^{m-2}u_0^{\prime\prime}.
\end{align}
and
\begin{align}\label{5-6}
Su_0^p|u_0^{\prime}|^q\leq Sc_0^p(-u_0^{\prime})^{m-1-p},
\end{align}
where we have used that $S$ is an increasing function.

Since $u_{0}(r)$ is a solution to
\begin{align*}
(m-1)u^{\prime\prime}+c_0^p(-u^{\prime})^{1-p}=0.
\end{align*}
Hence, combining (\ref{5-5}) and (\ref{5-6}), we know $u_{0}(r)$ is a solution to (\ref{rad}) near origin point.

Now we choose $c_0$ and $R_0$ to glue these two parts at $r=R_0$.
To obtain this, we need
\begin{align*}
\left\{
\begin{array}{ll}
u_{\eta}(R_0)=u_{0}(R_0), \\
u_{\eta}^{\prime}(R_0)=u_{0}^{\prime}(R_0).
\end{array}
\right.
\end{align*}
which is equivalent to
\begin{align}\label{5-8}
\left\{
\begin{array}{ll}
e^{-\eta R_0}=c_0-\frac{c_0(pR_0)^{\frac{p+1}{p}}}{(p+1)(m-1)^{\frac{1}{p}}}, \\
\\
-\eta e^{-\eta R_0}=-c_0\left(\frac{pR_0}{m-1}\right)^{\frac{1}{p}}.
\end{array}
\right.
\end{align}
By Intermediate Value theorem, it is easy to verify there exists a $R_0$ satisfying (\ref{5-00}) such that
\begin{align*}
\left(\frac{pR_0}{m-1}\right)^{\frac{1}{p}}=\eta-\frac{\eta(pR_0)^{1+\frac{1}{p}}}{(p+1)(m-1)^{\frac{1}{p}}},
\end{align*}
then it also follows the existence of $c_0$ by (\ref{5-8}).

Under the above choice of $R_0$ and $c_0$, let us define
\begin{align*}
u(r)=\left\{
\begin{array}{ll}
u_{0}(r), \quad r\in [0, R_0], \\
u_{\eta}(r), \quad r\in (R_0, +\infty).
\end{array}
\right.
\end{align*}
Noting that $u^p|\nabla u|^q\in L^{1}(B_{R_0}(o))$,
it follows that $u$ is a weak positive solution to (\ref{equ}).

\textbf{(2). Case of $p=0, q=m-1$.}
For the solution near infinity, let us define
\begin{align*}
u_{\eta}=e^{-\eta r},\quad\mbox{$r\geq R_0$}.
\end{align*}
One can verify that when
\begin{align}\label{5-iota}
\iota>1,
\end{align}
and
\begin{align*}
0<\eta\leq\frac{\iota-1}{m-1}
\end{align*}
then $u_{\eta}(r)$ is a solution to (\ref{rad}) near infinity.

For the solution near origin point, let us define
\begin{equation*}
u_{0}(r)=-r^\alpha+c_0,\quad\mbox{$r\leq R_0$},
\end{equation*}
where $\alpha>1$, $c_0$ are to be chosen later, and
\begin{align*}
0< R_0< \min\{c_0^{\frac{1}{\alpha}}, (m-1)(\alpha-1)\}.
\end{align*}
It is easy to verify that $u_0$ is a solution to (\ref{rad}) near origin.

Repeating the same procedure, we obtain there exists $R_0$, $c_0$ and $\alpha$ such that
\begin{align*}
u_{\eta}(R_0)=u_0(R_0),\quad u_{\eta}^{\prime}(R_0)=u_0^{\prime}(R_0).
\end{align*}
Let
\begin{align*}
u(r)=\left\{
\begin{array}{ll}
u_{0}(r), \quad r\in [0, R_0], \\
u_{\eta}(r), \quad r\in (R_0, +\infty).
\end{array}
\right.
\end{align*}
then $u$ is a weak positive solution to (\ref{equ}). Hence, we complete the proof of Theorem \ref{thm-ex} \textbf{(E)}.
\end{proof}

\begin{remark}
Let us compare $\iota$ defined in the proof of Theorem \ref{thm-ex} \textbf{(E)} and $\kappa$ introduced in Theorem \ref{thm-non} \textbf{(V)}.
Since $\iota$ satisfies
\begin{align*}
\iota\left\{
\begin{array}{ll}
>1, & p=0, q=m-1,\notag\\
\geq H(p), & p>0, q=m-1-p,
\end{array}
\right.
\end{align*}
and $\kappa$ satisfies
\begin{align*}
0<\kappa<\frac{\min\{m-1,1\}}{2C_1e}<\frac{\min\{m-1,1\}}{2e},
\end{align*}
where $H(p)$ is defined in (\ref{5-3-iota}), and we also have used that $C_1>1$.

Noting that $H(p)$ is increasing when $p\in[0,m-1)$, and decreasing when $p\in(m-1,+\infty)$, we obtain
$H(p)$ reaches maximum $m$ when $p=m-1$. Moreover,
\begin{align*}
\lim\limits_{p\downarrow0}H(p)=1,\quad \lim\limits_{p\to\infty}H(p)=m-1.
\end{align*}
It follows that $\kappa<\iota$ in both cases of $p=0$ and $p>0$, see Figure \ref{graph3}.


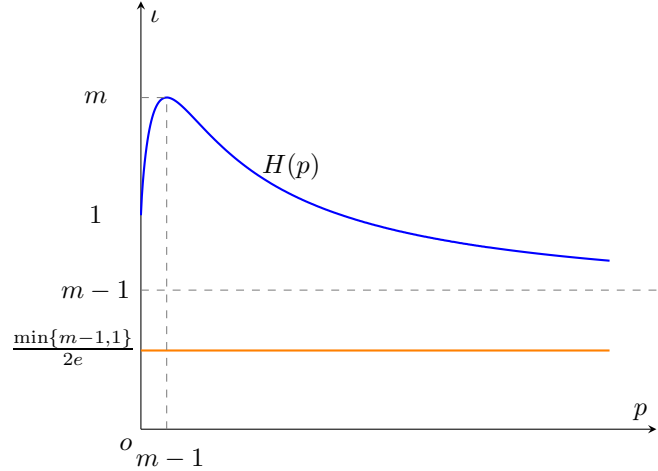
\begin{figure}[h]
\def \m {1.55}
\begin{tikzpicture}
\begin{axis}[
axis x line=middle,
axis y line=middle,
xlabel=$p$,
ylabel=$\iota$,
xmin=0, xmax=11, xmajorgrids=false, xtick={0},
ymin=0, ymax=2, ymajorgrids=false, ytick={0}]

\addplot[blue, thick, domain=0:10, samples=500] {(((\m-1)/x)^(x/(x+1)))*(1+x)};
\addplot[orange, thick, domain=0:10, samples=50] {1/e};
\addplot[dashed, gray, domain=0:11, samples=50] {\m-1+0.1};
\addplot[dashed, gray, domain=0:\m-1, samples=50] {\m};
\addplot[dashed, gray, domain=0:\m, samples=50] (\m-1,x);
\end{axis}
\node at (-0.6, 4.42) {$m$};
\node at (-0.6, 2.86) {$1$};
\node at (-0.6, 1.85) {$m-1$};
\node at (-0.9, 1.1) {$\frac{\min\{m-1, 1\}}{2e}$};
\node at (0.39, -0.39) {$m-1$};
\node at (-0.2,-0.2) {$o$};
\node at (2,3.5) {$H(p)$};
\end{tikzpicture}
\caption{$\kappa<\iota$ }
\label{graph3}
\end{figure}
\end{remark}

\vskip3ex
\begin{proof}[Proof of Theorem \ref{thm-ex} \textbf{(F)}]
We work on the model manifolds as before but requiring that
\begin{align*}
S(r)=e^{\lambda r^{\gamma}\ln r},\quad\mbox{$r\geq R_0$},
\end{align*}
and
\begin{align*}
V(r)\lesssim e^{\lambda r^{\gamma}\ln r},
\end{align*}
where $\lambda$ is any positive number, and $R_0$ is to be chosen later.

One can verify that when
$$\gamma>2(m-1-q)+1,$$
the function
\begin{align*}
u_{\infty}(r)=e^{\frac{\ln r}{r}},
\end{align*}
is a solution to (\ref{rad}) near infinity. Moreover, $u_{\infty}(r)$ is decreasing.

Repeating the same procedure as in the proof of Theorem \ref{thm-ex} \textbf{(D)}, we derive that there exists a nontrivial positive solution to (\ref{equ}). Hence, we complete the proof of Thereom \ref{thm-ex} \textbf{(F)}.
\end{proof}

\textbf{Acknowledgements}\;
The authors would like to thank Dr. Qingsong Gu (Nanjing University), and Prof. Xueping Huang (Nanjing University of Information Science $\&$ Technology) for helpful communication.

\bibliographystyle{amsalpha}

\end{document}